\documentclass[11pt,a4paper,twoside]{amsart}

\usepackage{amssymb} 
\usepackage{amsmath}
\usepackage{amsthm}
\usepackage{dsfont, mathrsfs}
\usepackage{color}
\usepackage{graphicx}

\usepackage{verbatim}
\usepackage{url}
\newtheorem{thm}{Theorem}[section]
\newtheorem{lem}{Lemma}[section]

\newtheorem{cor}{Corollary}[section]

\usepackage{caption}
\usepackage{array} 
\newcolumntype{M}[1]{>{\centering\arraybackslash}m{#1}} 
\usepackage{stackengine} 

\def\1{\mathds{1}}

\newcommand{\ve}{\varepsilon}
\newcommand{\set}[1]{\{ #1 \}}
\newcommand{\vp}{\varphi}
\newcommand{\un}{{\bf 1}}
\newcommand{\abs}[1]{\left\lvert #1 \right\rvert}

\newcommand{\epv}{\quad ; \quad}

\renewcommand{\gcd}{}

\newenvironment{proofbold}[1][\proofname]{\noindent{\textbf{Proof of #1.\ }}}{\hfill$\blacksquare$}

\begin{document}

\title[M\"obius function and primes: an identity factory]{M\"obius function and primes: an identity factory with applications}

\author[O. Ramar\'e]{Olivier Ramar\'e}
\address{CNRS / Institut de Math\'ematiques de Marseille, Aix Marseille Universit\'e,
U.M.R. 7373, Campus de Luminy, Case 907, 13288 Marseille Cedex 9,
France.}
\email{olivier.ramare@univ-amu.fr}
\urladdr{https://ramare-olivier.github.io/}

\author[S. Zuniga-Alterman]{Sebastian Zuniga-Alterman\smallskip\\
\MakeLowercase{with an appendix by} Michel Balazard}
\address{Department of Mathematics and Statistics, University of
  Turku, 20014 Turku, Finland.}
\email{szualt@utu.fi}

\address{CNRS / Institut de Math\'ematiques de Marseille, Aix Marseille Universit\'e,
U.M.R. 7373, Campus de Luminy, Case 907, 13288 Marseille Cedex 9, France.}
\email{michel.Balazard@math.cnrs.fr}

\date{}


\thanks{The first author has been partly supported by the joint
  FWF-ANR project Arithrand: FWF: I 4945-N and ANR-20-CE91-0006ANR.}
  
  \thanks{The second author has been supported by the Finnish Centre of Excellence in Randomness
and Structures (Academy of Finland grant no. 346307).}
\thanks{Thanks are due to F. Daval for his insightful comments and remarks on our work. We also thank the referees for their detailed remarks on this work.}

\begin{abstract}
  We investigate the sums
  $\sum_{n\le X, (n,q)=1}\frac{\mu(n)}{n^s}\log^k\left(\frac{X}{n}\right)$, where $k\in\{0,1\}$, 
  $s\in\mathbb{C}$, $\Re s>0$. Our goal is to obtain explicit asymptotic estimations for these quantities. To achieve this, we develop a broad framework of identities that we use to derive several applications. Building on similar principles, we also provide an appendix establishing the inequality $\sum_{n\le
    X}\Lambda(n)/n\le \log X$, valid for any $X\geq 1$.
\end{abstract}

\maketitle

\section{Introduction and results}

Estimates concerning the M\"obius function $\mu$ form a subject of independent interest, examined for instance by R.A. MacLeod in \cite{MacLeod*69},
by L.~Schoenfeld
in~\cite{Schoenfeld*69}, by N.~Costa Pereira in~\cite{CostaPereira*89},
by F. Dress \& M. El Marraki in~\cite{Dress*93},
\cite{Dress-ElMarraki*93} and \cite{ElMarraki*96},
in~\cite{Cohen-Dress-ElMarraki*96} with the help of H.~Cohen and
recently by K.A.~Chalker in the memoir~\cite{Chalker*19}.  We refer to the survey paper \cite{Ramare*18-9} by the first
author for an overview of the various developments in the field.

A common strategy for obtaining estimates involving the M\"obius is to link $\mu$ to
the distribution of primes, as in \cite{Ramare*12-2}, or to rely on certain identities, as in
\cite{Schoenfeld*69} by L.~Schoenfeld. The identity approach
gained momentum after the paper \cite{Balazard*12} by M.~Balazard, and has
been particularly effective to obtain estimates for the M\"obius-related sum $\sum_{n\leq X} \mu(n)/n$,
with or without coprimality conditions; see for example the PhD
memoir \cite{Daval*19} by F.~Daval or the paper \cite{Ramare*12-5}. In
this latter, the role of coprimality conditions is clearly
exposed. A different (and numerically more efficient) treatment of coprimality conditions was introduced
by H.~Helfgott \cite{Helfgott19}, used for instance in Lemma~3.3 of
\cite{Alterman*22b} by the second author. See also the paper
\cite{Camargo*25} by A.~Camargo.

In this article, we explore the identity approach further by obtaining an {\em identity factory} (see Theorem \ref{OFD}). As an application, we derive explicit estimations of the following two M\"obius-related quantities
\begin{equation}
  \label{defmdve}
  m_{q}(X;s)
  =\sum_{\substack{n\le X,\\ (n,q)=1}}\frac{\mu(n)}{n^{s}},
  \quad
  \check{m}_{q}(X;s)
  =\sum_{\substack{n\le X,\\ (n,q)=1}}\frac{\mu(n)}{n^{s}}\log\left(\frac{X}{n}\right),
\end{equation}
where either $s\in\mathbb{C}$ and $\Re s>0$, or $s=1+\ve$, with $\ve>0$. 

The expression $m_{q}(X;s)$ is of intrinsic significance since, at $s=1$, the Prime Number Theorem tells us that $\sum_n\mu(n)/n=0$. Thus, estimating the partial sums $\sum_{n\leq X}\mu(n)/n$ is a central problem in analytic number theory. For instance, the bound $\sum_{n\leq X}\mu(n)/n=O(X^{-1/2+\ve})$, $\ve>0$, is equivalent to the Riemann Hypothesis. Similarly, $  \check{m}_{q}(X;s)$ is a logarithmically smoothed version of $m_{q}(X;s)$ and arises naturally, since the logarithm function appears once we derivate the Dirichlet series associated with $m_{q}(X;s)$. 

\subsection*{\textbf{Notation}}

Throughout the present work the variable $p$ denotes a prime
number. We use the notation
$f(X)=O^*(h(X))$ to indicate that $|f(X)|\leq h(X)$.  Consider now
$q,d\in\mathbb{Z}_{>0}$; we write $d|q^\infty$ to mean that $d$ is in
the set $\{d',\ p|d'\implies p|q\}$.  We use $[x]$ to denote the
integer part of the real number $x$, and $\{x\}$ to denote its
fractional part, so that $x=[x]+\{x\}$.  Finally, we consider the {\em
  generalised Euler totient function} (also called {\em Jordan's totient function}): let~$s$ be any complex
number, we define $\varphi_s:\mathbb{Z}_{>0}\to\mathbb{C}$ by
$q\mapsto q^s\prod_{p|q}\left(1-\frac{1}{p^s}\right)$. The constant
function that takes uniformly the value~1 is denoted by
$\mathds{1}$. We use also the notation $\mathds{1}_{\mathscr{P}}$ to denote the
characteristic function of those positive integers that satisfy the
condition $\mathscr{P}$.\\
  
  \subsection*{\textbf{The Identity Factory}}
Estimating  $m_{q}(X;s)$ and $\check{m}_{q}(X;s)$ via the identity approach succeeds even in the complex case $s\in\mathbb{C}$. The theorem below, which is proved  in \S\ref{POFD},
forms the core of
this article. It provides a flexible and powerful  
framework that unifies and extends previous identity-based approaches (see $\S 3$). 
\begin{thm}[The Identity Factory]
  \label{OFD}
  Let $f$ and $g$ be two arithmetic functions.
  We define $S_f(t)=\sum_{n\le t}f(n)$ and $S_{f\star g}(t)$ similarly.
  Let $h:(0,1]\rightarrow \mathbb{C}$ be Lebesgue-integrable 
  over every segment $\subset (0,1]$ and let $H$ be a function over
  $[1,\infty)$ that is absolutely
  continuous on every finite interval of $[1,\infty)$.  When $X\ge1$, we have
  \begin{multline}\label{star}
  \qquad\qquad\sum_{n\le X}f(n)H\left(\frac{X}{n}\right)-H(1)\sum_{n\le X}f(n)
  = 
  \\ \int^{X}_1 S_{f\star g}\Bigl(\frac{X}{t}\Bigr)h\left(\frac{1}{t}\right)\frac{dt}{t}+\int_1^X S_f\Bigl(\frac{X}{t}\Bigr)
  \biggl(H'(t)-\frac{1}{t}\sum_{n\le t}g(n)h\left(\frac{n}{t}\right)\biggr)
  dt.
  \end{multline}
\end{thm}

\subsection*{\textbf{A wide-ranging estimate}}
Concerning $\check{m}_{q}(X;s)$, with $s\in\mathbb{R}$, we obtain an important estimation. By
exploiting the Taylor expansion and
bounds near~$s=1$, we derive explicit bounds in the real neighbourhood of~$s=1$. Although stronger
input is required, this approach succeeds in proving the next result.
\begin{thm}\label{nonnegative}
  \label{easy} 
  When $k\ge1$ is an integer, $\sigma\ge1$ and $X\ge1$, we have
  \begin{equation*}
  0\le \sum_{\substack{n\le X\\ (n,q)=1}}\frac{\mu(n)}{n^{\sigma}}
  \log^k\biggl(\frac{X}{n}\biggr)
  \le
  1.00303\frac{q}{\varphi(q)}\bigl(
  k
  +(\sigma-1) \log X
  \bigr)(\log X)^{k-1},
\end{equation*}
where $\varphi$ is the Euler totient function. When $k\neq 1$, we may replace $1.00303$ by 1.
\end{thm}
The non-negativity of the above expression is noteworthy. While the partial sums of $\sum_{n\leq X}\mu(n)$ and $\sum_{n\leq X}\mu(n)/n$ oscillate infinitely often as $X$ varies, proved for instance by J. Kaczorowski and J. Pintz in \cite[Cor. 4]{Kaczorowski-Pintz*86}, the smoothed version  $\sum_{n\leq X}\mu(n)\log(X/n)/n$ is always non-negative, as shown by Theorem \ref{nonnegative}. See \cite{Polya*19} and \cite[\S 25]{Turan*48} for historical discussion.

\subsection*{\textbf{Recovering previous results through the Identity Factory}}
In $\S$\ref{POFDbis} we show that choosing $H(t)=t$ and $g(n)=1$ in Theorem \ref{OFD}
recovers earlier known identities (most of them, as documented in
in~\cite{Ramare*18-9}).

Moreover, the identity path also yields efficient estimates
for $\sum_{p\le X}\log p/p$ from estimates of
$\sum_{p\le X}\log p$ with the choice
$f=\Lambda$, the von-Mangoldt function. More background information on this question may be found
in~\cite{Ramare*12-1}. 

To illustrate further the importance of non-trivial identities, we devote the appendix to prove the following result.
\begin{thm}
  \label{Harmonique}
Let $X \ge 1$. Then we have
$\displaystyle
\sum_{n \le X} \frac{\Lambda(n)}n \le \log X$.\\
\end{thm}

\subsection*{\textbf{Obtaining new identities via the Identity Factory}} The intuition behind Theorem \ref{OFD} is to choose $H$ and $g$ so that $H'(t)-1/t\sum_{n\le t}g(n)h\left(n/t\right)$ and $S_{f\star g}(X/t)$, behave like error terms, so that the left hand side of \eqref{star} is well-approximated.

The quantity $(1/t)\sum_{n\le t}h(n/t)$, being a Riemann sum,
is approximately constant when~$t$ is large; thus the
choice $g(n)=1$ almost forces the choice $H'(t)=\int_0^1 h(u)du$ if we want cancellations to
occur. Likewise, for fixed $f$, this trivial choice of $g$ ensures that $S_{f\star g}(X/t)$ will not be as small as we need. Thus, it is relevant that we have at the same time more freedom to choose the function $H$ as well as that $g$ oscillates as an arithmetic function.

Nonetheless, the rigidity in selecting the function $H$ disappears once we select, for instance, $g(n)=(-1)^{n+1}$. At the
same time, this new function $g$ is {\em close} to a convolution inverse of~$\mu$, as for instance shown in Lemma~\ref{thisconvol1}. This
ensures that the integral containing $S_{f\star g}$ remains a
negligible quantity.

Using these ideas, we obtain explicit estimates of the
two quantities displayed in \eqref{defmdve}. We use the shortcuts $m_{q}(X;1)=m_q(X)$,
$\check{m}_q(X;1)=\check{m}_q(X)$ and we shall further omit the index
when $q=1$. We shall distinguish the case when $s$ is real and larger
than $1$ (completing Theorem~\ref{easy}) and the general case.

\subsection*{Expressing $m_q(X;s)$ and $\check{m}_q(X;s)$ in terms of
  $m_q(t)$ for general $s$}
Let us start with two rather general expressions.
\begin{thm}\label{mqdex}
  Let $X\geq 1$, $q\in\mathbb{Z}_{>0}$ and $s\in\mathbb{C}$ such that
  $\zeta(s)\neq 0$, $s\not\in Z=\{1+2\pi i k/\log 2, k\in\mathbb{Z}_{\neq 0}\}$ and $\Re s=\sigma\ge\sigma_0>0$ for some fixed $\sigma_0$. Then we have
\begin{align*}
  &\Bigg|\sum_{\substack{n\le X,\\ (n,q)=1}}
  \frac{\mu(n)}{n^{s}}-\frac{m_q(X)}{X^{s-1}}
  -\frac{q^s}{\varphi_s(q)}\frac{1}{\zeta(s)}\Bigg|
   \\
   &\phantom{xxxxxxxxxxx}\leq\ 
     \frac{\sigma+|s|}{\sigma |c(s)\zeta(s)|X^{\sigma}}
     \int_1^X|m_q(t)|dt
     +
     \frac{|c(s)|+2^{\sigma_0}e(s)}{|c(s)\zeta(s)|X^{\sigma_0}}
     \frac{q^{\sigma-\sigma_0}}{\varphi_{\sigma-\sigma_0}(q)},
   \end{align*}
   where $c(s)=\frac{1-2^{1-s}}{s-1}$ and
   $e(s)=2^{1-\sigma}(1+2^{|s-1|-1}|s-1|\log 2)\log
   2$. Notice that $c(1)=\log 2$. Furthermore, when $s=\sigma$ is
   real, the factor $(\sigma+|s|)/\sigma$ may be replaced by~1.
 \end{thm}
 
Observe that the above right-hand side tends to $0$ when $s\to 1$, as expected.

Here is the counterpart of Theorem~\ref{mqdex} for~$\check{m}_q(X;s)$.
\begin{thm}\label{mcheckqdex} 
 Let $X\geq 1$, $q\in\mathbb{Z}_{>0}$ and $s\in\mathbb{C}$ such that
  $\zeta(s)\neq 0$, $s\not\in Z=\{1+2\pi i k/\log 2, k\in\mathbb{Z}_{\neq 0}\}$ and $\Re s=\sigma\ge\sigma_0>0$ for some fixed $\sigma_0$. Then 
\begin{multline*}
  \Biggl|
  \sum_{\substack{n\le X,\\ (n,q)=1}}
  \frac{\mu(n)}{n^{s}}\log\left(\frac{X}{n}\right)
  -
  \frac{q^s}{\varphi_s(q)}
  \biggl(\frac{\log X}{\zeta(s)}-\frac{\zeta'(s)}{\zeta^2(s)}
  -\frac{1}{\zeta(s)}\sum_{p|q}\frac{\log(p)}{p^s-1}\biggr) 
  \Biggr|
  \\\le
  \frac{\Xi_1(s)}{X^{\sigma}}\int_1^X
  |m_q(t)|dt
  +\frac{\Xi_2(X;s,\sigma_0)}{X^{\sigma_0}}
  \frac{q^{\sigma-\sigma_0}}{\varphi_{\sigma-\sigma_0}(q)},
\end{multline*}
where 
\begin{align}\label{Xi1}
  \Xi_1(s)
  &=\frac{(\sigma+|s|)((\sigma+|s-1|)|C(s)|
    +\sigma|s-1||C'(s)|)}{\sigma^2 C(s)^2},\\
\label{Xi2}
  \Xi_2(X,\sigma_0,s)
  &=\frac{2^{\sigma_0}\Bigl(\log X+\delta\Bigl(\frac{X}{2},\sigma_0\Bigr)
    \max\left\{\log\Bigl(\frac{X}{2}\Bigr),
    \frac{1}{\sigma_0}\right\}\Bigr)}{|\zeta(s)|}\nonumber\\
  &+\frac{2^{\sigma_0}\log 2}{|c(s)\zeta(s)|}
    +\frac{2^{\sigma_0}}{|\zeta(s)|}
    \left|\frac{C'(s)}{C(s)}-\frac{1}{(s-1)}-\frac{\zeta'(s)}{\zeta(s)}\right|\nonumber\\
  +\frac{2^{\sigma_0}}{|\zeta(s)|}
  &\left|\frac{C'(s)}{C(s)}-\frac{1}{(s-1)}\right|
    +2^{\sigma_0}\frac{e(s)}{|C(s)|}
    \left|\frac{1}{C(s)}+\frac{(s-1)C'(s)}{C(s)}\right|.
\end{align}
We have set $C(s)=(1-2^{1-s})\zeta(s)$, 
$\delta(\frac{X}{2},\sigma_0)
=1+\mathds{1}_{\log({X}/{2})<{1}/{\sigma_0}}$,
while $c(s)$ and~$e(s)$ are defined as in Theorem \ref{mqdex}. Moreover, if $s=\sigma>0$, then the factor $(\sigma+|s|)/\sigma$ appearing in $\Xi_1(s)$ may be replaced by $1$.
\end{thm}

The integral of~$|m_q(t)|$ appearing in the above two results is
treated in Lemma~\ref{boundmqdex}. Moreover, the case $q=1$ and
$s=\sigma > 1$ is particularly significant.
\begin{thm}
  \label{special}
  For any $\sigma\in[1,1.04]$, we have the following estimation
\begin{align*}
  &\left|\sum_{n\leq X}\frac{\mu(n)}{n^{\sigma}} \log\left(\frac{X}{n}\right)
    -\frac{\log X}{\zeta(\sigma)}+\frac{\zeta'(\sigma)}{\zeta^2(\sigma)}
    \right|\\
  &\phantom{xxxxxxxxxxxxxxxxxxx}\leq
    \frac{15.5+ 3.11(\sigma-1) \log X}{X^{\sigma-\frac{1}{2}}},
  &&\text{ if }15\leq X\leq 10^{14},\\
  &\phantom{xxxxxxxxxxxxxxxxxxx}\leq\frac{0.043}{X^{\sigma-1}\log X},
  &&\text{ if }X\geq 10^{14}.
\end{align*}
\end{thm}
 
 \subsection*{The Identity Factory vs Summation by Parts} In general, the Identity Factory (Theorem \ref{OFD}) offers significant advantages over a direct application of summation by parts. For instance, for any $s\not\in Z$, applying
 integration by parts to compare $1/n^s$ with
 $1/n$ produces an error term of size
 $|s-1|\int_{1}^X|m_q(t)|dt/t^\sigma$, whereas Theorem \ref{mqdex} yields an error term of size
 $\int_{1}^X|m_q(t)|dt/[X^\sigma|\zeta(s)|]$ (note that $c(s)\gg 1$ for $s\not\in Z$). 
 
 Observe that this latter
 error magnitude gives greater weight to the large values of~$t$, leading to
 improved error terms. To illustrate this, take $q=1$ and assume the
 Riemann Hypothesis, so that $m(t)\ll_\ve t^{-\frac12+\ve}$ for every
 $\ve>0$. If we further assume that $\frac{1}{2}-\sigma+\ve>c$, for some constant $c>0$, then applying integration by parts
 leads to an error term of order $O(|s-1|)$, whereas Theorem \ref{mqdex} gives the stronger
  estimation $O(|s-1|X^{-(\frac{1}{2}-\sigma+\ve)})$, provided that $\frac{1}{2}-\sigma+\ve>0$.
\subsection*{Expressing $m_q(X;s)$ and $\check{m}_q(X;s)$ in terms of
  $m_q(t)$ for $s\in[1,2]$}
Let us prove more explicit versions of the two above results when $s$
is a real number close to $1^+$. 
\begin{thm}
 \label{mqeps} Let $q\in\mathbb{Z}_{>0}$.
  Let $\sigma=1+\ve\in[1,2]$ and $X\geq 1$, we have the following estimation
  \begin{equation*}
    m_q(X;\sigma)
    =\frac{m_q(X)}{X^{\sigma-1}}
    +\frac{q^{\sigma}}{\varphi_{\sigma}(q)}\frac{1}{\zeta(\sigma)}
    +\frac{(\sigma-1)\Delta_q(X,\sigma-1)}{X^{\sigma-1}}
  \end{equation*}
  where
  \begin{equation*}
    |\Delta_q(X,\ve)|\le
  0.03\,g_1(q)\frac{q^\xi}{\varphi_{\xi}(q)}\frac{\1_{X\ge 10^{12}}}{\log(X)}
    +
    \biggl(
    4.1\,g_0(q)
    +
    \frac{(5 + \ve 2^\ve)}{2}
    \biggr)\frac{\sqrt{q}}{\varphi_{\frac{1}{2}}(q)}\frac{2^\ve}{\sqrt{X}},
  \end{equation*}
  where, with $\xi=1-{1}/({12\log 10})$, the functions $g_0$ and $g_1$
  are defined by:
   \begin{equation} \label{defg0g1}
     g_0:q\in\mathbb{Z}_{>0}\mapsto
     \begin{cases}
       \frac{\sqrt{6}-\sqrt{3}}{2}&\text{if $2|q$},\\
       1&\text{else},
     \end{cases}
          \hspace*{8pt}
          g_1:q\in\mathbb{Z}_{>0}\mapsto
\begin{cases}
       1.4378\left(1-\frac{1}{2^\xi}\right)&\text{if $2|q$},\\
       1&\text{else}.
     \end{cases}
  \end{equation}
  
  Moreover, we have $m_q(X;\sigma)\ge m_q(X)/X^{\sigma-1}$. With respect to $\Delta_q(X,\ve)$, we have the following upper bounds
  \begin{equation*}
    \left\{
      \begin{array}{ll}
        \Delta_q(X,\ve) \le 0 &\text{ if } X < 10.85,\\
        \Delta_q(X,\ve) \le 0 &\text{ if } X < 10.9,\  \ve \le 4/25,\\
        \Delta_q(X,\ve)/X^\ve \le 0 &\text{ if } X < 41,\ \gcd\bigl(q,\prod_{p\le
                                      37}p\bigr)\notin\{1,11,13\},\\
        \Delta_q(X,\ve)/X^\ve \le 0 &\text{ if } X < 47,\ \gcd\bigl(q,\prod_{p\le
                                      43}p\bigr)\notin\{1,11,13, 17\},\\
        \Delta_q(X,\ve)/X^\ve \le 0.014 &\text{ if } X \le 47,\ \gcd\bigl(q,\prod_{p\le
                                          43}p\bigr)=1,\\
        \Delta_q(X,\ve)/X^\ve \le 0.00005 &\text{ if } X < 47,\ \gcd\bigl(q,\prod_{p\le
                                             43}p\bigr)\in\{11,13,
                                            17\}.
      \end{array}
    \right.
  \end{equation*}
  Finally, we also have  the lower bound
  $\Delta_q(X,\ve)/X^\ve\ge -q/\varphi(q)$. This lower bound may be
  refined to $-\frac{1}{\ve\zeta(1+\ve)}\frac{q^{1+\ve}}{\varphi_{1+\ve}(q)}$.
\end{thm}
Observe that the function $\Delta_q(X,\ve)$ is positive when $q=1$,
$X=10.97$ or $X=11$, and $\sigma\in[1,2]$. The condition $X<47$ is set
only to keep the running verification time within an acceptable bound,
as we have to range over all the divisors of $\prod_{p<47}p$. For more details, refer to \S\ref{discrete}.

\begin{thm}
  \label{mcheckqeps}
  Let $\sigma=1+\ve\in[1,11/10]$ and $X\geq 15$, we have the estimate
\begin{equation*}
  \check{m}_q(X;\sigma)
  =
     \frac{q^{\sigma}}{\varphi_{\sigma}(q)}
     \biggl(\frac{\log(X)}{\zeta(\sigma)}
     -\frac{\zeta'(\sigma)}{\zeta^2(\sigma)}
     -\frac{1}{\zeta(\sigma)}\sum_{p|q}\frac{\log
       p}{p^{\sigma}-1}\biggr) +\frac{\check{\Delta}_q(X,\sigma-1)}{X^{\sigma-1}}
 \end{equation*}
 where we have
 \begin{align*}
   |\check{\Delta}_q(X,\ve)|
   \le
   &0.0336g_1(q)2^\ve\frac{q^\xi}{\varphi_{\xi}(q)}\frac{\1_{X\geq 10^{12}}} {\log(X)}
     \\&+(4.86 g_0(q)+2.93 + 2.83 \ve\log(X)
     +5.17\ve)\frac{\sqrt{q}}{\varphi_{\frac{1}{2}}(q)}\frac{2^\ve}{\sqrt{X}},
 \end{align*}
 where $g_0$ and $g_1$ are defined in~\eqref{defg0g1}, and $\xi$, just
 above.
\end{thm}

At $\sigma=1$, the obtained result is
comparable to but weaker than \cite[Lemma~3.3]{Alterman*22b}.
 \subsection*{The Identity Factory and Secondary Order Terms} Theorems \ref{mqdex}, \ref{mcheckqdex} and \ref{mqeps} are significant because they provide non-trivial asymptotic results (even in a non-explicit form), and more importantly, they also reveal the presence of secondary order terms. 
 
 To illustrate this, suppose we ignore all secondary order terms. Then for any $\sigma=1+\varepsilon>1$,
\begin{align*}
  m(X;\sigma)
  =\sum_{\substack{n\le X}}\frac{\mu(n)}{n^{1+\varepsilon}}&=\frac{1}{\zeta(1+\varepsilon)}+O\left(\sum_{\substack{n> X}}\frac{\mu^2(n)}{n^{1+\varepsilon}}\right)\\
  &=\frac{1}{\zeta(1+\varepsilon)}+O\left(\int_X^\infty\frac{dt}{t^{1+\varepsilon}}\right)
  =\frac{1}{\zeta(1+\varepsilon)}+O\left(\frac{1}{\varepsilon X^\varepsilon}\right),
\end{align*}
so that, by letting $\varepsilon\to 0^+$, the error term diverges.

Thus, any meaningful asymptotic result should aim for an estimation such that
\begin{equation*}
  m(X;\sigma)=m(X;1+\varepsilon)\to m(X)=\sum_{\substack{n\le X}}\frac{\mu(n)}{n}\quad\text{ as }\quad\varepsilon\to 0^+,
\end{equation*}
which cannot be achieved unless secondary order terms are handled with care.

\subsection*{\textbf{A methodological remark}}

The treatment of error terms in analytic number theory is of utmost
importance. The Perron summation formula highlights the use of
complex analysis and multiplicative characters, whereas the exponential-sum
method essencially relies on the Fourier expansion of the sawtooth
function $x\mapsto\{x\}-\frac12$, and thus brings complex analysis in
the additive world.  In contrast, our approach is based on real analysis,
handling the error terms directly by absolute value bounds. In
the language of Theorem \ref{OFD}, these error terms correspond to 
$H'(t)-\frac{1}{t}\sum_{n\le t}g(n)h\left(\frac{n}{t}\right)$ and the quantity $S_{f\star g}(X/t)$, where
we are usually selecting either $g(n)=1$ or $g(n)=(-1)^{n+1}$.
Precisely, a key novelty of this article is the use
of~$g(n)=(-1)^{n+1}$ rather than relying solely on the choice $g(n)=1$.

\section{The Identity Factory, proof of~Theorem~\ref{OFD}}
\label{POFD}

\begin{proofbold}[Theorem~\ref{OFD}]
  On the one hand, by the local absolute continuity of $H$, $H$ has a
  derivate almost everywhere, which is Lebesgue-integrable. Thus, by
  integration by parts, we obtain
  \begin{equation*}
      \int_1^X S_f\Bigl(\frac{X}{t}\Bigr)H'(t)dt
      =\sum_{n\le X}f(n)H\left(\frac{X}{n}\right)-H(1)S_f(X).
  \end{equation*}
  On the other hand, we have
  \begin{align*}
          \int_1^X S_f\Bigl(\frac{X}{t}\Bigr)
    \frac{1}{t}\sum_{n\le t}g(n)h\left(\frac{n}{t}\right)
  dt
  &=
  \sum_{n\le X}g(n)\int_n^X S_f\Bigl(\frac{X}{t}\Bigr)
    h\left(\frac{n}{t}\right)
  \frac{dt}{t}
  \\&=
  \sum_{n\le X}g(n)\int_1^{\frac{X}{n}} S_f\left(\frac{{X}/{n}}{t}\right)
    h\left(\frac{1}{t}\right)
  \frac{dt}{t}
  \\&=
  \int_1^{X} \sum_{n\le{X}/{t}}g(n)S_f\left(\frac{{X}/{t}}{n}\right)
    h\left(\frac{1}{t}\right)
  \frac{dt}{t},
  \end{align*}
  where we have used summation by parts, a change of variables and
  then Fubini's theorem, respectively. The proof follows on noticing
  the identity
  \begin{equation*}
  \sum_{n\le {X}/{t}}g(n)\sum_{m\le \frac{X}{tn}}f(m)=\sum_{\ell\le {X}/{t}}(f\star g)(\ell)
  \end{equation*}
 valid  for any real number $X\ge1$.
\end{proofbold}

\section{Recovering earlier results through the Identity Factory}
\label{POFDbis}
The case $H(t)=t$, $g=\mathds{1}$ and $f=\mu$ in Theorem~\ref{OFD} yields the
following statement, which is the initial result of \cite{Daval*19}.
\begin{cor}
  \label{FD}
  Let $h:(0,1]\rightarrow \mathbb{C}$ be any Lebesgue-integrable function
  over every segment of $(0,1]$.  When $X\ge1$, we have
  \begin{equation*}
  m(X)-\frac{M(X)}{X}
  =
  \frac{1}{X}\int_{\frac{1}{X}}^1\frac{h(t)}{t}dt
  +\frac{1}{X}\int_1^X M\left(\frac{X}{t}\right)
  \biggl(1-\frac{1}{t}\sum_{n\le t}h\left(\frac{n}{t}\right)\biggr)
  dt.
  \end{equation*}
\end{cor}
Although not required, it is better to normalize $h$ by imposing the
condition $\int_0^1h(t)dt=1$.
In \cite[Theorem 7.4]{Ramare*19}, it is proven that one recovers all
the (regular enough) identities linking $m(X)$ and $M(X)$ with the
result of Corollary~\ref{FD}, so that the above is not only a
curiosity that is included in a further stream of identities.

Let us see two more examples.
\begin{cor}\label{FD1}
  When $X\ge1$,
  we have $\displaystyle \int_1^X\bigg[\frac{X}{t}\bigg] \frac{M(t)}{t}dt=\log X$.
\end{cor}
This can be also found in \cite[\S 1]{ElMarraki*96} and
    in~\cite[Eq. (9.2)]{Ramare*12-2}.
\begin{proof}
  We select $h=\mathds{1}$ in Corollary~\ref{FD}.
  We first obtain
  \begin{equation*}
    Xm(X)-M(X)=\log X+\int_1^X M\left(\frac{X}{t}\right)\bigg(1-\frac{[t]}{t}\bigg)dt.
  \end{equation*}
  On noticing that
  \begin{equation*}
    \int_1^XM\bigg(\frac{X}{t}\bigg)dt=\sum_{n\le X}\mu(n)\int_1^{X/n}dt=Xm(X)-M(X),
  \end{equation*}
  and performing a change of variables, the claimed identity follows.
\end{proof}
\begin{cor}\label{FD2}
  When $X\ge1$,
  we have 
  \begin{equation*}
    \sum_{n\leq X}\mu(n)
    \frac{\left\{{X}/{n}\right\}^2-\left\{{X}/{n}\right\}}{{X}/{n}}
    =Xm(X)-M(X)-2+\frac{2}{X}.
  \end{equation*}
\end{cor}
This identity is due to MacLeod in~\cite{MacLeod*94}. More identities of this type together
with an efficient theoretical background to handle them may be found in the
paper~\cite[\S 5]{Balazard*12b} by~Balazard. The next proof is in
a large part extracted from this source.
\begin{proof}
  We select $h(t)=2t$. For $t>0$, Define
  \begin{equation}
    \label{defalpha}
    \alpha(t) = \frac{1-2\set{t}}t-\frac{\set{t}-\set{t}^2}{t^2} ,\qquad t>0.
  \end{equation}
  On one hand, we have
  \begin{equation}\label{230910a}
    \frac 2{X^2}\sum_{n \le X} n=1+\alpha(X) , \qquad X>0.
  \end{equation}
  On the other hand, we readily check that $\alpha$ is precisely the
  right derivate of the function 
  \begin{equation}
    \label{defbeta}
    \beta : t \mapsto \frac{\set{t}-\set{t}^2}t
  \end{equation}
  over $(0,\infty)$.
  We next notice that
  \begin{equation*}
    1-\frac{1}{t}\sum_{n\le t}\frac{2n}{t}
    =-\alpha(t).
  \end{equation*}
  We are ready to apply Corollary~\ref{FD}. Thus, we derive
  \begin{align*}
    Xm(X)-M(X)
    &=2-\frac{2}{X}
      -\int_1^X M\bigg(\frac{X}{t}\bigg)\alpha(t)dt
    \\&=
    2-\frac{2}{X}
    -\sum_{n\le X}\mu(n)\int_1^{X/n}\alpha(t)dt
    \\&=
    2-\frac{2}{X}
    +\sum_{n\le X}\mu(n)\frac{\{X/n\}^2-\{X/n\}}{X/n}
  \end{align*}
  as required.
\end{proof}

The functional transform that, to a function $h$, associates the
function $X>0\mapsto\int_0^1h(t)dt-\frac{1}{X}\sum_{n\le
  X}h(\frac{n}{X})$ is closely related to a transform introduced by Ch.~M\"untz in
\cite{Muntz*22}. This is also discussed by E.~Titchmarsh in \cite[Section
2.11]{Titchmarsh*86} and more information can be found
in~\cite{Yakubovich*15} by S.~Yakubovich. 

\begin{cor}
  When $X\ge1$, we have
  \begin{multline*}
      \sum_{n\le X}\frac{\mu(n)}{n}\log\left(\frac{X}{n}\right)
      +\gamma \biggl(\sum_{n\le X}\frac{\mu(n)}{n}-\frac{M(X)}{X}\biggr)
      \\=1-\frac{1}{X} +\frac{1}{X}\int_1^X
      M\left(\frac{X}{t}\right)\biggl(\log t +\gamma+\frac1t-\sum_{n\le t}\frac1n\biggr)dt,
  \end{multline*}
  where $\gamma$ is Euler's constant.
\end{cor}

\begin{proof}
We select, for $t$ a real number larger than~$1$, $H(t)=t\log t
-\log t+\ \gamma t$ and
$h(t)=1/t$, and $f=\mu$ and $g=\mathds{1}$ in Theorem~\ref{OFD}. 
   We readily derive
   \begin{multline*}
      \sum_{n\le X}\frac{\mu(n)}{n}\log\left(\frac{X}{n}\right)
      -\sum_{n\le X}\frac{\mu(n)}{X}\log\left(\frac{X}{n}\right)
      +\gamma \biggl(\sum_{n\le X}\frac{\mu(n)}{n}-\frac{M(X)}{X}\biggr)
      \\=1-\frac{1}{X}
      +\frac{1}{X}\int_1^X M\left(\frac{X}{t}\right)
      \biggl(\log t+\gamma-\sum_{n\le t}\frac1n\biggr)dt.
  \end{multline*}
 The result is obtained by observing that, by summation by parts, we have
   \begin{equation*}
     \sum_{n\le X}\mu(n)\log\left(\frac{X}{n}\right)
     =\int_1^X M\left(\frac{X}{t}\right)\frac{dt}{t}.
  \end{equation*}
\end{proof}
Here is a novel corollary.
\begin{cor}
  Let $\lambda(n)=(-1)^{\Omega(n)}$ denote the Liouville function.
  For every $X\ge1$, we have
  \begin{multline*}
    \sum_{n\le X}\frac{\lambda(n)}{n}
    -\frac{1}{X}\sum_{n\le
      X}\lambda(n)=
    \\
    \frac{2}{\sqrt{X}}-\frac{1}{X}
    -\frac{1}{X}\int_1^X\left\{\frac{X}{t}\right\}\frac{dt}{t}
    + \frac{1}{X}\int_1^X \sum_{n\le {X}/{t}}\lambda(n)\{t\}\frac{dt}{t}.
  \end{multline*}
\end{cor}

\begin{proof}
  Apply Theorem~\ref{OFD} with $H(t)=t$, $h=\mathds{1}$,
  $f=\lambda$ and $g=\mathds{1}$. With this choice, observe that
  $S_{f\star g}\left(X/t\right)=\left[X/t\right]$. The proof is complete
\end{proof}


\section{Proof of Theorem~\ref{easy}}

In \cite[Cor. 1.10]{Ramare*12-5}, we find the next lemma.
\begin{lem}
  \label{Ram2}
  For any $q\in\mathbb{Z}_{>0}$ and any $X>0$, we have
  \begin{equation*}
    0\le \sum_{\substack{n\le X\\ (n,q)=1}}\frac{\mu(n)}{n}\log
    \left(\frac{X}{n}\right)\le 1.00303\frac{q}{\varphi(q)}.
  \end{equation*}
\end{lem}

In \cite[Prop. A.4, p. 126]{Srivasta*19} by P. Srivasta, we find the following.
\begin{lem}
  \label{Prim}
  For any $q\in\mathbb{Z}_{>0}$, any $X>0$, and any integer $k\ge2$, we have
  \begin{equation*}
    0\le \sum_{\substack{n\le X\\ (n,q)=1}}\frac{\mu(n)}{n}\log^k\biggl(
    \frac{X}{n}\biggr)\le k\frac{q}{\varphi(q)}\log^{k-1}(X)
  \end{equation*}
\end{lem}
\begin{proof}
  The case $k=2$ is proved in \cite[Cor. 1.11]{Ramare*12-5}. The
  case $k\ge3$ is readily deduced from this one by summation by parts.
\end{proof}

\begin{proofbold}[Theorem~\ref{easy}] If $\sigma=1$, then we obtain the result thanks to Lemma \ref{Prim}. We may suppose then that $\sigma-1=\ve>0$.
  On using the expansion
  \begin{equation*}
    \biggl(\frac{X}{n}\biggr)^\ve=\exp\bigl(\ve\log(X/n)\bigr)
    =\sum_{\ell\ge 0}\frac{\ve^\ell}{\ell!}\log^\ell\biggl(\frac{X}{n}\biggr),
  \end{equation*}
  we deduce that
\begin{equation*}
  X^\ve\sum_{\substack{n\le X\\ (n,q)=1}}\frac{\mu(n)}{n^{1+\ve}}
  \log^k\biggl( \frac{X}{n}\biggr)
  =
  \sum_{\ell\ge 0}\frac{\ve^\ell}{\ell!}
  \sum_{\substack{n\le X\\ (n,q)=1}}\frac{\mu(n)}{n}
  \log^{k+\ell}\biggl( \frac{X}{n}\biggr).
\end{equation*}
The above exchange of summations is justified by absolute convergence. When $k\ge1$, by combining lemmas~\ref{Ram2} and \ref{Prim}, this implies that
\begin{align*}
  0&\le X^\ve\sum_{\substack{n\le X\\ (n,q)=1}}\frac{\mu(n)}{n^{1+\ve}}
  \log^k\biggl(\frac{X}{n}\biggr)
  \le 1.00303\frac{q}{\varphi(q)}
  \sum_{\ell\ge 0}\frac{\ve^\ell\log^{k+\ell-1}(X)(k+\ell)}{\ell!}
  \\&\le
  1.00303\frac{q}{\varphi(q)}\biggl(
  k\log^{k-1}(X)\exp(\ve \log X)
  +\ve \log^k (X)\exp(\ve\log X)
  \biggr),
\end{align*}
from which the theorem follows.
\end{proofbold}

\section{The function $m_q$}
\subsection*{Auxiliaries on $m_q$}\label{Moeb}

In \cite[Cor. 1.2]{Ramare-Zuniga*25-1}, we have proved the following result
\begin{lem}
  \label{update}
  For $X\ge 617\,990$, we have
  \begin{equation*}
    \Bigg|\sum_{n\le X}\frac{\mu(n)}{n}\Bigg|
    \le      \frac{0.010032\log X-0.0568}{\log^2 X}.
  \end{equation*}
\end{lem}

Thus, for any $X\geq 1$ and $q\in\mathbb{Z}_{q>0}$, by following \cite[Lemma 5.10]{Helfgott19} and \cite[Prop. 5.15]{Helfgott19}, we have
\begin{equation}
  \label{coprimality}
  \biggl|\sum_{\substack{n\leq X\\(n,q)=1}}\frac{\mu(n)}{n}\biggr|
  \leq
  \frac{\sqrt{2}\sqrt{q}}{\varphi_{\frac{1}{2}}(q)}\frac{1}{\sqrt{X}}
  +\frac{0.010032\ q^{\theta}}{\varphi_{\theta}(q)}\frac{\mathds{1}_{X\geq 10^{14}}}{\log X},
\end{equation}
where $\theta=1-\frac{1}{14\log 10}$. 

Moreover, when $q=2$, we have \cite[Eq. (5.79) \& (5.89)]{Helfgott19}
at our disposal, namely
\begin{align}
  \label{eq:2}
      |m_2(X)|&\le \sqrt{\frac{3}{X}},\quad\text{ if }0<X\leq 10^{12},\nonumber
      \\|m_2(X)|&\le \frac{0.0296}{\log X},\quad\text{ if }X\geq 5379,
\end{align}
Similar to \eqref{coprimality}, when $(q,2)=1$ and $\xi =1-\frac{1}{12\log 10}$, we have, for any $X>0$,
\begin{equation}
  \label{cp2}
  \biggl|\sum_{\substack{n\leq X\\(n,2q)=1}}\frac{\mu(n)}{n}\biggr|
  \leq
  \frac{\sqrt{3}\sqrt{q}}{\varphi_{\frac{1}{2}}(q)}\frac{1}{\sqrt{X}}
  +\frac{0.0296q^{\xi }}{\varphi_{\xi }(q)}\frac{\mathds{1}_{X\geq 10^{12}}}{\log X}.
\end{equation}
The above bound may we written as follows.
\begin{lem}
  \label{basemq}
  For any $q\in\mathbb{Z}_{>0}$ and $X>0$, we have
  \begin{equation*}
    \biggl|\sum_{\substack{n\leq X\\(n,q)=1}}\frac{\mu(n)}{n}\biggr|
    \leq
    \frac{g_0(q)\sqrt{q}}{\varphi_{\frac{1}{2}}(q)}\frac{\sqrt{2}}{\sqrt{X}}
    +\frac{0.010032\ g_1(q)q^{\xi}}{\varphi_{\xi}(q)}\frac{\mathds{1}_{X\geq 10^{12}}}{\log X},
  \end{equation*}
  the multiplicative functions $g_0$ and $g_1$ being defined in~\eqref{defg0g1}.
\end{lem}
The value at small values of the parameter $X$ are often of crucial
impact while the slight worsening of the second term has much less effect.
\begin{proof}
  We may assume $q$ to be squarefree.
  When $q$ is odd, this is a slight degrading of~\eqref{coprimality}.
  When $q$ is even, this is a consequence of~\eqref{cp2} since
  $ 0.010032\, g_1(2)\le 0.0296$.
\end{proof}
\subsection*{The integral of $|m_q|$}
Let us first recall \cite[Lemma~7.1]{Ramare*12-5}.
\begin{lem}
  \label{gene}
  Let $A>e$ be a given parameter. The function 
  \begin{equation*}
    T:y\mapsto\frac{\log y}{y}\int_{A}^y\frac{dt}{\log t}
  \end{equation*}
  is first increasing and then decreasing. It reaches its maximum at $y_0(A)$ where
  $y_0(A)$ is the unique solution of
  $y=(\log y-1)\int_A^y{dt}/{\log t}$. Moreover we have $T(y_0(A))=(\log
  y_0(A))/(\log y_0(A)-1)$.
\end{lem}
With Lemma \ref{basemq} at hand, we
derive the following result. 
\begin{lem}
  \label{boundmqdex}
  For any $X\geq 1$ and $q\in\mathbb{Z}_{q>0}$, we have
  \begin{equation*}
    \int_1^X|m_q(t)|dt
    \le \frac{0.010333\ g_1(q)q^{\xi}}{\varphi_{\xi}(q)}
    \frac{X\1_{X\geq 10^{12}} }{\log X}
    + \frac{g_0(q)\sqrt{q}}{\varphi_{\frac{1}{2}}(q)}\sqrt{8X},
  \end{equation*}
  where the multiplicative functions $g_0$ and $g_1$ are defined in~\eqref{defg0g1}.
\end{lem}

\begin{proof}
  Note that it is enough to consider the case when $q$ is
  squarefree. Indeed the functions $m_q(t)$, $g_1(q)$, $g_0(q)$, $q^\xi/\vp_\xi(q)$ and $
  \sqrt{q}/\vp_{1/2}(q)$ are all depending only on the squarefree part of~$q$. 
  We distinguish two cases. If $X< 10^{12}$, then by Lemma~\ref{basemq}, we have
 \begin{equation*}
    \int_1^X|m_q(t)|dt \le
     \frac{g_0(q)\sqrt{q}}{\varphi_{\frac{1}{2}}(q)}\int_1^X\frac{\sqrt{2}}{\sqrt{t}}dt\leq
     \frac{g_0(q)\sqrt{q}}{\varphi_{\frac{1}{2}}(q)}\sqrt{8X}. 
  \end{equation*}
  On the other hand, we first use the following Pari/GP script
\begin{verbatim}
  g(y)=(log(y)-1)*intnum(t=10^(12), y, 1/log(t));
  solve(y=10^12, 10^16, g(y)-y),
\end{verbatim}
  which tells us that the value $y(10^{12})$ defined in Lemma~\ref{gene} corresponds to
  $y_0(10^{12})=1365396548134370.8\cdots$. We have $T(y_0(10^{12}))\le 1.03$.
  Therefore, by combining Lemma~\ref{basemq} and Lemma~\ref{gene}, we have
  \begin{align*} 
    \int_1^X|m_q(t)|dt
    &\leq
       \frac{g_0(q)\sqrt{q}}{\varphi_{\frac{1}{2}}(q)}\int_1^{X}\frac{\sqrt{2}}{\sqrt{t}}dt
      +\frac{g_1(q)q^{\xi}}{\varphi_{\xi}(q)}\int_{10^{12}}^X\frac{0.010032}{\log t}dt
    \\ &\leq
       \frac{g_0(q)\sqrt{q}}{\varphi_{\frac{1}{2}}(q)}\sqrt{8X}
      +\frac{0.010333\ g_1(q)q^{\xi}}{\varphi_{\xi}(q)}\frac{X}{\log X},
    \end{align*}
    whence the result.
\end{proof}

\section{Evaluating $m_q(X;s)$, proof of Theorem~\ref{mqdex}}

\begin{lem}
  \label{thisconvol1}
  Let $X>0$ and $q\in\mathbb{Z}_{>0}$. Consider the arithmetic function
  $g_1:n\in\mathbb{Z}_{>0}\mapsto(-1)^{n+1}$. We have the following identities
  \begin{align*}
      (i)&\qquad\sum_{n\le X}g_1(n)=\mathds{1}_{([X],2)=1}(X),\\
    (ii)&\qquad G_1(n)=\sum_{\substack{d_1d_2=n\\ (d_1,q)=1}}\mu(d_1)g_1(d_2)
    =\mathds{1}_{n|q^\infty}(n)
    -2\cdot \mathds{1}_{2|n,\frac{n}{2}|q^\infty}(n),\\
    (iii)&\qquad\sum_{n\le X}\frac{G_1(n)}{n}=
    \sum_{\substack{\frac{X}{2}<\ell\le X\\ \ell|q^\infty}}\frac{1}{\ell}.
  \end{align*}
\end{lem}
\begin{proof} The definition of $g_1$ gives $(i)$.
  On the other hand, the Dirichlet series of $g_1$ is $(1-2^{1-s})\zeta(s)$ while the one
  of the $\mathds{1}_{(n,q)=1}(n)\mu(n)$ is $\prod_{p\nmid
    q}(1-p^{-s})$. Their product satisfies
  \begin{equation*}
    (1-2^{1-s})\prod_{p|q}(1-p^{-s})^{-1}
    =(1-2^{1-s})\sum_{\ell|q^\infty}\frac{1}{\ell^s}
    =\sum_{\ell|q^\infty}\frac{1}{\ell^s}
    -\sum_{\substack{2|\ell\\ \frac{\ell}{2}|q^\infty}}\frac{2}{\ell^s},
  \end{equation*}
   which gives the convolution identity $(ii)$. We readily obtain $(iii)$ from $(ii)$.
\end{proof}
Let us see next a useful result concerning alternating sums.

\begin{lem}\label{alternating} $(i)$ Let $(u_n)_{n\geq 1}$ be a positive real decreasing sequence such that $\sum_{n} (-1)^{n+1} u_n $ converges. Then for any $Y>0$, we have 
\begin{equation*}\sum_{n > Y} (-1)^{n+1} u_n=O^*(u_{[Y]+1}).
\end{equation*}
$(ii)$ Let $(v_n)_{n\geq 1}$ be a positive real increasing sequence. Then, for any $X>Y>0$, we have the following estimation
\begin{equation*}\sum_{Y<n\leq X} (-1)^{n+1} v_n=O^*(v_{[X]}).
\end{equation*}
$(iii)$ Let $(w_n)_{n\geq 1}$ be a positive real sequence that increases up to a value $X>0$ and then decreases, and such that $\sum_{n} (-1)^{n+1} w_n $ converges. Then, for any $X>Y>0$, we have 
\begin{equation*}\sum_{n>Y} (-1)^{n+1} w_n=O^*(w_{[X]}).
\end{equation*}
\begin{equation*}.
\end{equation*}

\begin{proof} $i)$ Let $N=[Y]$. Since $(u_n)$ is decreasing we have that
\begin{align*}&(-1)^N\sum_{n > Y} (-1)^{n+1} u_n=(-1)^N\sum_{n > N} (-1)^{n+1} u_n\\
&\qquad=u_{N+1}-(u_{N+2}-u_{N+3})-(u_{N+4}-u_{N+5})-...\leq u_{N+1}.
\end{align*}
On the other hand, we have that
\begin{align*}&(-1)^N\sum_{n > N} (-1)^{n+1} u_n=(u_{N+1}-u_{N+2})+(u_{N+3}-u_{N+4})+...\geq 0.
\end{align*}
The proof for the case $(ii)$ is similar. Finally, in order to derive case $(iii)$, we split the corresponding series at $[X]$, giving a sum of type $(ii)$ and a series of type $(i)$, and notice that they have opposite signs.
\end{proof}
\end{lem}
\begin{lem}
  \label{coyote}
  When $\Re s=\sigma> 0$ and $X>0$, we have
  \begin{equation*}
    \sum_{n\le X}\frac{g_1(n)}{n^{s}}=\sum_{n\le X}\frac{(-1)^{n+1}}{n^{s}}
    =C(s)+O^*\biggl(\frac{\sigma+|s|}{\sigma X^{\sigma}}\biggr).
  \end{equation*}
  with $C(s)=(1-2^{1-s})\zeta(s)$, $C(1)=\log 2$. When $s=\sigma$ is
  real,
  the error term reduces to $O^*(1/X^\sigma)$.
\end{lem}
\begin{proof}
  Observe that
  \begin{align}\label{equ1}
     \sum_{n\le X}\frac{(-1)^{n+1}}{n^{s}}
    &=
      \sum_{n\le X}(-1)^{n+1}\biggl(\frac{1}{X^{s}}+s\int_{n}^X\frac{dt}{t^{s+1}}\biggr) \nonumber
    \\&=
    \frac{\mathds{1}_{([X],2)=1}(X)}{X^{s}}
    +s
    \int_{1}^X\frac{\mathds{1}_{([t],2)=1}(t)}{t^{s+1}}dt,
  \end{align}
  where we used Lemma~\ref{thisconvol1} $(i)$ and then Fubini's theorem. By letting $X\to\infty$, we can write \eqref{equ1} as
$C(s)+E(s)$,
  where   \begin{align*}
    C(s)&=s\int_{1}^\infty\frac{\mathds{1}_{([t],2)=1}(t)}{t^{s+1}}dt=\sum_{n}\frac{(-1)^{n+1}}{n^{s}},\\
    |E(s)|&=\left|\frac{\mathds{1}_{([X],2)=1}(X)}{X^{s}}
    -s\int_{X}^\infty\frac{\mathds{1}_{([t],2)=1}(t)}{t^{s+1}}dt\right|\leq\frac{\sigma+|s|}{\sigma X^{\sigma}}.
  \end{align*}
  The integral expression of $C(s)$ shows that it is a holomorphic function whose corresponding alternating series converges for $\sigma>0$. By analytic continuation, we have that $C(s)=(1-2^{1-s})\zeta(s)$ for any $\sigma>0$. Moreover, for any $z\geq 1$, the function
  $C_z:s\in\mathbb{C}\mapsto(1-z^{1-s})\zeta(s)$ is entire and satisfies
  $C_z(1)=\log z$.
  
  Finally, when $s=\sigma>0$, we encounter an alternating series given by $C(\sigma)$, which, thanks to Lemma~\ref{alternating} $(i)$, gives a better bound $|E(\sigma)|\leq X^{-\sigma}$.
\end{proof}

\begin{lem}
  \label{approx}
  Let $X>0$ and $s\in\mathbb{C}$. Let $e(s)=2^{1-\sigma}(1+2^{|s-1|-1}|s-1|\log 2)\log 2$. Then, if $\Re s =\sigma\ge\sigma_0>0$, we have
  \begin{align*}
    (i)&\Bigg|\sum_{\substack{\ell>X,\\ \ell| q^\infty}}\frac{1}{\ell^{s}}\Bigg|
    \le\frac{1}{X^{\sigma_0}} \frac{q^{\sigma-\sigma_0}}{\varphi_{\sigma-\sigma_0}(q)}\\
     (ii)&\Bigg|\sum_{\substack{\frac{X}{2}<\ell\leq X\\
    \ell|q^{\infty}}}
    \frac{2^{1-s}-\left(\frac{X}{\ell}\right)^{1-s}}{s-1}\frac{1}{\ell^s}
    \Bigg|
    \leq
    \frac{2^{\sigma_0}e(s)}{X^{\sigma_0}}\frac{q^{\sigma-\sigma_0}}{\varphi_{\sigma-\sigma_0}(q)}.
  \end{align*}
\end{lem}

\begin{proof}
  Observe that $\sum_{\substack{\ell|
      q^\infty}}\frac{1}{\ell^{\omega}}$ converges to
  $\frac{q^\omega}{\varphi_\omega(q)}$ for any $\omega\in\mathbb{C}$
  such that $\Re\omega>0$. Thus, as $\sigma-\sigma_0>0$, 
\begin{equation*}
    \Bigg|\sum_{\substack{\ell>X,\\ \ell| q^\infty}}\frac{1}{\ell^{s}}\Bigg|
    \le \frac{1}{X^{\sigma_0}} \sum_{\ell|
      q^\infty}\frac{1}{\ell^{\sigma-\sigma_0}}
    =\frac{1}{X^{\sigma_0}} \frac{q^{\sigma-\sigma_0}}{\varphi_{\sigma-\sigma_0}(q)},
  \end{equation*}
  whence $(i)$.
  
  On the other hand, we use
  $2^{1-s}-(X/\ell)^{1-s}=2^{1-s}(1-(X/(2\ell))^{1-s}$. Recall that,
  for any $z\in[0, 1]$,
  we have
  $(z^{1-s}-1)(1-s)^{-1}=\log z +O^*(2^{-1}|s-1|z^{-|s-1|}\log^2z)$. Therefore,
  by taking $  z=\frac{X}{2\ell}\le 1$ and using that $\sigma_0>0$, we
  derive $(ii)$.  
\end{proof}

\begin{proofbold}[Theorem~\ref{mqdex}] We use Theorem~\ref{OFD}
  with $h:t\in(0,1]\mapsto(s-1) t^{1-s}C(s)^{-1}$,
  $H:t\in[1,\infty)\mapsto t^{s-1}$,
  $g:n\in\mathbb{Z}_{>0}\mapsto \frac{(-1)^{n+1}}{n}$ and
  $f(n)=\mathds{1}_{(n,q)=1}\mu(n)/n$, where
  $C(s)=(1-2^{s-1})\zeta(s)$. Note that since $s\not\in Z$,
  $1-2^{s-1}\neq 0$ unless $s=1$. Nevertheless, $C$ is holomorphic at
  $s=1$, with value $C(1)=\log 2\neq 0$. Thus, as $\zeta(s)\neq 0$
  too, $h$ is well-defined.
  
  Moreover, by Lemma~\ref{thisconvol1} $(ii)$, we have
  $(f\star g)(n)=G_1(n)/n$. Therefore, by
  Lemma~\ref{thisconvol1} $(iii)$ and Lemma~\ref{coyote}, we may
  express
  \begin{equation}\label{s:ident}
    \sum_{\substack{n\le X,\\ (n,q)=1}}\frac{\mu(n)}{n^{s}}-\frac{m_q(X)}{X^{s-1}}=  
    \frac{(s-1)}{C(s)}M_1(X;q,s,\sigma_0)+O^*\left(\frac{R_1(X;q,s)}{X^{\sigma}}\right),
   \end{equation}
   where, on considering the holomorphic function
   $c:s\in\mathbb{C}\mapsto\frac{1-2^{1-s}}{s-1}$, $c(1)=\log 2$, and
   on recalling Lemma~\ref{approx}, we have
\begin{align*}
  M_1(X;q,s,\sigma_0)
  =
  \int_1^{X} \biggl(\sum_{\substack{\frac{t}{2}<\ell\le t\\
      \ell|q^\infty}}\frac{1}{\ell}\biggr)
  \frac{dt}{t^{s}}
  =\nonumber
  \sum_{\substack{\ell\le X,\\ \ell|q^\infty}
  }\frac{1}{\ell}\int_{\ell}^{\min(2\ell,X)}\frac{dt}{t^{s}}&
  \\=
  \frac{1}{s-1}\sum_{\substack{\ell\le X,\\ \ell|q^\infty}}\nonumber
  \frac{1}{\ell^{s}}\left(1-\min\left(2,\frac{X}{\ell}\right)^{1-s}\right)&
  \\=
  \frac{1}{s-1}\sum_{\substack{\ell\le \frac{X}{2},\\ \ell|q^\infty}}\nonumber
  \frac{1-2^{1-s}}{\ell^{s}}+\frac{1}{s-1}\sum_{\substack{
  \frac{X}{2}<\ell\leq X,\\
  \ell|q^\infty}}\frac{1}{\ell^s}\left(1-\left(\frac{X}{\ell}\right)^{1-s}\right)& 
\end{align*}
so that
\begin{align}\label{M1:est}
  M_1(X;q,s,\sigma_0)
  &=
  \frac{1}{s-1}\sum_{\substack{\ell\le X,\\ \ell|q^\infty}}\nonumber
  \frac{1-2^{1-s}}{\ell^{s}}
  +\frac{1}{s-1}
  \sum_{\substack{ \frac{X}{2}<\ell\leq X,\\ \ell|q^\infty}}
  \frac{1}{\ell^s}\left(2^{1-s}-\left(\frac{X}{\ell}\right)^{1-s}\right)
  \\&=
  c(s)\frac{q^s}{\varphi_s(q)}
  +
  O^*\left(
  \frac{|c(s)|+2^{\sigma_0}e(s)}{X^{\sigma_0}}
  \frac{q^{\sigma-\sigma_0}}{\varphi_{\sigma-\sigma_0}(q)}\right),&
\end{align}
and where, 
\begin{align}\label{R1:est}
  R_1(X;q,s)&= X\int_1^X m_q\left(\frac{X}{t}\right)(s-1)t^{s-2}\left(1-\frac{1}{C(s)}\sum_{n\leq t}\frac{g_1(n)}{n}\right)dt\nonumber \\
  &= X\int_1^X m_q\left(\frac{X}{t}\right) O^*\biggl(\frac{(\sigma+|s|)|s-1|}{\sigma |C(s)|}\biggr)\frac{dt}{t^2}\nonumber\\
  &=O^*\left(\frac{(\sigma+|s|)|s-1|}{\sigma |C(s)|}\int_1^X|m_q(t)|dt\right). 
\end{align}  
When $s=\sigma>0$, the above factor $(\sigma+|s|)/\sigma$ may be replaced by $1$.
Finally, by estimations \eqref{M1:est}, \eqref{R1:est} and by observing that $\frac{s-1}{C(s)}=\frac{1}{c(s)\zeta(s) }$, we deduce from \eqref{s:ident} that
\begin{align*}
    \Bigg|\sum_{\substack{n\le X,\\ (n,q)=1}}\frac{\mu(n)}{n^{s}}-\frac{m_q(X)}{X^{s-1}}-\frac{q^s}{\varphi_s(q)}\frac{1}{\zeta(s)}\Bigg|\leq\frac{R_1(X;q,s)}{X^{\sigma}}+\frac{R_2(X;s,\sigma_0)}{X^{\sigma_0}}\frac{q^{\sigma-\sigma_0}}{\varphi_{\sigma-\sigma_0}(q)},
   \end{align*}
   where
   \begin{equation}  \label{R2:est}
   R_2(X;s,\sigma_0)=\frac{|c(s)|+2^{\sigma_0}e(s)}{|c(s)\zeta(s)|}.
 \end{equation}
 The theorem is proved.
\end{proofbold}

\section{Evaluating $\check{m}_q(X;s)$, proof of Theorem~\ref{mcheckqdex}}
\begin{lem}
  \label{groscoyote}
  When $\Re s=\sigma>0$ and $X>0$, we have
  \begin{equation*}
    \sum_{n\le X}\frac{(-1)^{n+1}}{n^{s}}\log\left(\frac{X}{n}\right)
    =C(s)\log X+C'(s)
    +O^*\Bigl(\frac{\sigma+|s|}{\sigma^2 X^{\sigma}}\Bigr).
  \end{equation*}
  where $C'(s)$ is the derivative of $C(s)=(1-2^{1-s})\zeta(s)$ with
  respect to~$s$. When $s=\sigma$ is real, the error term is
  $O^*(1/(e\sigma X^\sigma))$.
  \end{lem}
\begin{proof}
  By recalling Lemma~\ref{thisconvol1} $(i)$ and using summation by parts, we observe that for any $Y>0$,
  \begin{align}
    \label{coyote2}
    &\sum_{n\le Y}\frac{(-1)^{n+1}}{n^{s}}\log\left(\frac{X}{n}\right)
    =
       \sum_{n\le
      Y}(-1)^{n+1}\Biggl(\frac{\log\left(\frac{X}{Y}\right)}{Y^{s}}+
      \int_{n}^Y\frac{1+s\log\left(\frac{X}{t}\right)}{t^{s+1}}dt\Biggr)\nonumber
    \\&\phantom{xxxxx}=
    \frac{\mathds{1}_{([Y],2)=1}\log\left(\frac{X}{Y}\right)}{Y^{s}}
    +
    \int_{1}^Y \frac{\mathds{1}_{([t],2)=1}\left(1+s\log\left(\frac{X}{t}\right)\right)}{t^{s+1}}dt    \\&
    \label{coyote3}=
     \frac{\mathds{1}_{([Y],2)=1}\log\left(\frac{X}{Y}\right)}{Y^{s}}
    +B(X,s)
    -
    \int_Y^\infty\frac{\mathds{1}_{([t],2)=1}\left(1+s\log\left(\frac{X}{t}\right)\right)}{t^{s+1}}dt
  \end{align}
  where
  \begin{equation}
    \label{eq:1}
    B(X,s)=\int_{1}^\infty
    \frac{\mathds{1}_{([t],2)=1}\left(1+s\log\left(\frac{X}{t}\right)\right)}{t^{s+1}}dt.
  \end{equation}
  Suppose that $\Re s>1$. Then we see that
  \begin{equation*}
    \sum_{n=1}^\infty\frac{(-1)^{n+1}}{n^s}=C(s),
    \quad
    -\sum_{n=1}^\infty\frac{(-1)^{n+1}\log n}{n^s}=C'(s).
  \end{equation*}
  Therefore, by letting $Y\to\infty$ in \eqref{coyote2}, we obtain 
  \begin{equation}\label{extension}
  B(X,s)=C(s)\log X+C'(s).
  \end{equation} 
  As $s\mapsto B(X,s)$ is holomorphic for $\Re s>0$, by analytic
  continuation, \eqref{extension} is valid for $\Re s>0$. Thereupon,
  by selecting $Y=X$ in \eqref{coyote3}, we obtain 
\begin{equation*}
    \sum_{n\le X}\frac{(-1)^{n+1}}{n^{s}}\log\left(\frac{X}{n}\right)
    =C(s)\log X+C'(s)
    +O^*\Bigl(\frac{\sigma+|s|}{\sigma^2 X^{\sigma}}\Bigr),
  \end{equation*}  
  valid for $\Re s=\sigma>0$. Indeed, the error term is bounded by noticing that
  \begin{align*}
  \Bigg|\int_X^\infty\frac{\mathds{1}_{([t],2)=1}\left(1+s\log\left(\frac{X}{t}\right)\right)}{t^{s+1}}dt\Bigg|\leq
     \int_X^\infty
    \frac{1+|s|\log\left(\frac{t}{X}\right)}{t^{\sigma+1}}dt\phantom{xxxxxxxxxx}&\\
    =\left[-\frac{1}{\sigma
    t^\sigma}-\frac{|s|}{\sigma}\left(\frac{\log\left(\frac{t}{X}\right)}{t^\sigma}+\frac{1}{\sigma
    t^\sigma}\right)\right|_X^\infty
    =\frac{\sigma+|s|}{\sigma^2X^\sigma}.&
  \end{align*}
  When $s=\sigma$ is a real number, we see that
  \begin{equation*}
    \sum_{n\le X}\frac{(-1)^{n+1}}{n^{s}}\log\left(\frac{X}{n}\right)
    -(C(s)\log X+C'(s))=\sum_{n>X}\frac{(-1)^{n+1}\log\left(\frac{n}{X}\right)}{n^\sigma}.
  \end{equation*}
  By computing the derivative of $t\mapsto \log(t/X)/t^\sigma$, we see that the sequence $\{\log(n/X)/n^\sigma\}_{n\geq 1}$ is increasing when $X<n\leq e^{1/\sigma}X$ and decreasing when $n>e^{1/\sigma}X$. Thus, by Lemma~\ref{alternating} $(iii)$, we obtain a better error magnitude $O^*(1/(e\sigma X^\sigma))$.  
  \end{proof}

\begin{lem}
  \label{approx2}
  Let $X>0$ and $s\in\mathbb{C}$. If $\Re s =\sigma>\sigma_0>0$, then
  \begin{align*}
    \sum_{\substack{\ell\leq X\\ \ell|
    q^\infty}}\frac{\log \ell}{\ell^{s}}
    =\frac{q^s}{\varphi_s(q)}\sum_{p|q}\frac{\log p}{p^s-1}
    +
    O^*\left(
    \frac{\delta(X,\sigma_0) \max(\log X, {1}/{\sigma_0})}
    {X^{\sigma_0}}\frac{q^{\sigma-\sigma_0}}{\varphi_{\sigma-\sigma_0}(q)}
    \right),
  \end{align*}
  where $\delta(X,\sigma_0)=1+\mathds{1}_{\log X<\frac{1}{\sigma_0}}(X)$.
\end{lem}
\begin{proof}
  As $\sum_{\substack{\ell|
      q^\infty}}{\ell^{-\omega}}=\frac{q^\omega}{\varphi_\omega(q)}$
  for any $\omega\in\mathbb{C}$ such that $\Re\omega>0$, we may
  differentiate that equality with respect to $\omega$. Thus,
\begin{equation}\label{diff}
  -\sum_{\substack{\ell| q^\infty}}\frac{\log \ell}{\ell^{\omega}}
  =
  \frac{\text{d}}{\text{d}\omega}
  \left(\prod_{p|q}\left(1-\frac{1}{p^\omega}\right)^{-1}\right)
  =
  -\frac{q^{\omega}}{\varphi_{\omega}(q)}\sum_{p|q}\frac{\log p}{p^\omega-1}.
\end{equation}
Therefore
\begin{equation*}
\sum_{\substack{\ell\leq X\\ \ell|
    q^\infty}}\frac{\log\ell}{\ell^{s}}
=
\frac{q^s}{\varphi_s(q)}\sum_{p|q}\frac{\log p}{p^s-1}
-\sum_{\substack{\ell>X,\\ \ell| q^\infty}}\frac{\log\ell}{\ell^{s}}.
\end{equation*} 
Furthermore, note that $t>0\mapsto(\log t)t^{-\sigma_0}$ is decreasing
for $t\geq e^{\frac{1}{\sigma_0}}$. Thus, if $X\geq
e^{\frac{1}{\sigma_0}}$, as $\sigma-\sigma_0>0$, we have 
\begin{equation*}
    \Bigg|\sum_{\substack{\ell>X\\ \ell| q^\infty}}\frac{\log \ell}{\ell^{s}}\Bigg|
    \le
    \frac{\log
      X}{X^{\sigma_0}}\sum_{\ell|q^\infty}\frac{1}{\ell^{\sigma-\sigma_0}}
    =
    \frac{\max(\log
      X,\frac{1}{\sigma_0})}{X^{\sigma_0}}\frac{q^{\sigma-\sigma_0}}
    {\varphi_{\sigma-\sigma_0}(q)};
    \end{equation*}
    whereas, if $X<e^{\frac{1}{\sigma_0}}$, we have
    \begin{align*}
    \Bigg|\sum_{\substack{\ell>X\\ \ell| q^\infty}}\frac{\log
      \ell}{\ell^{s}}\Bigg|
      \leq
    \sum_{\substack{\ell>e^{\frac{1}{\sigma_0}}\\ \ell|
      q^\infty}}\frac{\log\ell}{\ell^{\sigma}}
      +
      \sum_{\substack{X<\ell\leq e^{\frac{1}{\sigma_0}}\\ \ell|
      q^\infty}}\frac{\log \ell}{\ell^{\sigma}}
    \leq
      \left(\frac{1}{\sigma_0e}+\frac{1}{\sigma_0X^{\sigma_0}}\right)
      \sum_{\ell|q^\infty}\frac{1}{\ell^{\sigma-\sigma_0}}
      &\\
      \leq\frac{2}{\sigma_0X^{\sigma_0}}\frac{q^{\sigma-\sigma_0}}{\varphi_{\sigma-\sigma_0}(q)}
      =\frac{2\max(\log X,\frac{1}{\sigma_0})}{X^{\sigma_0}}
      \frac{q^{\sigma-\sigma_0}}{\varphi_{\sigma-\sigma_0}(q)},&
  \end{align*}
  whence the result.
  \end{proof}
  
  \begin{lem}
    \label{approx3}
    Let $q\in\mathbb{Z}_{>0}$. For any $X\geq 1$ we have the following estimation
  \begin{align*}
  &\phantom{xxxxxxxxxxxxxxxxxx}\int^{X}_1 \biggl(\sum_{\substack{\frac{t}{2}<\ell\le t\\
      \ell|q^\infty}}\frac{1}{\ell}\biggr)\log\left(\frac{X}{t}\right) \frac{dt}{t^{s}}=\\
  &\frac{q^s}{\varphi_s(q)}\biggl(c(s)\log X+c'(s)-c(s)\sum_{p|q}\frac{\log(p)}{p^s-1}\biggr)+O^*\left(\frac{R_3(X;s,\sigma_0)}{X^{\sigma_0}}\frac{q^{\sigma-\sigma_0}}{\varphi_{\sigma-\sigma_0}(q)}\right),
  \end{align*}
  where $c'(s)$ is the derivate of $c(s)=\frac{1-2^{1-s}}{s-1}$ with respect to $s$ and
  \begin{align}\label{R3:est}
  R_3(X;s,\sigma_0)=\ 2^{\sigma_0}\log 2\phantom{xxxxxxxxxxxxxxxxxxxxxx}&\nonumber\\
  +\ 2^{\sigma_0}\Biggl(|c(s)\log
    X+c'(s)|+|c(s)|\delta\Bigl(\frac{X}{2},\sigma_0\Bigr)
    \max\left(\log\Bigl(\frac{X}{2}\Bigr),\frac{1}{\sigma_0}\right)\Biggr).&
  \end{align}
  \end{lem}
  \begin{proof}
  Suppose first that $s\neq 1$. On exchanging integral and summation symbols (Fubini's theorem), we derive
\begin{align}\label{step3.1}
   &\int^{X}_1 \biggl(\sum_{\substack{\frac{t}{2}<\ell\le t\\
      \ell|q^\infty}}\frac{1}{\ell}\biggr)\log\left(\frac{X}{t}\right)
  \frac{dt}{t^{s}}
  =
  \sum_{\substack{
      \ell\le X\\ \ell|q^\infty}}\frac{1}{\ell}
  \int_{\ell}^{\min\{2\ell,X\}} \log\left(\frac{X}{t}\right)\frac{dt}{t^{s}}\nonumber
   \\&
     \phantom{xxxxxxxxx}=
  \sum_{\substack{
      \ell\le \frac{X}{2}\\ \ell|q^\infty}}\frac{1}{\ell}
  \int_{\ell}^{2\ell} \log\left(\frac{X}{t}\right)\frac{dt}{t^{s}}
  +\sum_{\substack{\frac{X}{2}<\ell\le X\\ \ell|q^\infty}}\frac{1}{\ell}
  \int_{\ell}^{X} \log\left(\frac{X}{t}\right)\frac{dt}{t^{s}}.
  \end{align}
  By Lemma~\ref{approx} $(i)$, we have
  \begin{align}\label{step3.2}
  \Bigg|\sum_{\substack{\frac{X}{2}<\ell\le X\\ \ell|q^\infty}}\frac{1}{\ell}
  \int_{\ell}^{X} \log\left(\frac{X}{t}\right)\frac{dt}{t^{s}}\Bigg|
  &\leq
  \sum_{\substack{\frac{X}{2}<\ell\le X\\ \ell|q^\infty}}\frac{1}{\ell}
  \int_{\ell}^{2\ell} \log\left(2\right)\frac{dt}{t^{\sigma}}\nonumber\\
  &\leq
  \sum_{\substack{\frac{X}{2}<\ell\\ \ell|q^\infty}}\frac{\log 2}{\ell^{\sigma}}\leq\frac{\log(2)2^{\sigma_0}}{X^{\sigma_0}}\frac{q^{\sigma-\sigma_0}}{\varphi_{\sigma-\sigma_0}(q)}.
  \end{align}
  On the other hand, by observing that
  $c'(s)=\frac{1}{s-1}(-c(s)+\log(2)2^{1-s})$ and using Lemma
  \ref{approx2}, we have
  \begin{align}\label{step3}
   &\sum_{\substack{
      \ell\le \frac{X}{2}\\ \ell|q^\infty}}\frac{1}{\ell}
  \int_{\ell}^{2\ell} \log\left(\frac{X}{t}\right)\frac{dt}{t^{s}}=
  \sum_{\substack{\ell\le \frac{X}{2}\\ \ell|q^\infty}}\frac{1}{\ell}\biggl(\frac{-c(s)+\log(2)2^{1-s}}{(s-1)}\frac{1}{\ell^{s-1}}+c(s)\frac{\log\left(\frac{X}{\ell}\right)}{\ell^{s-1}}\biggr)  \nonumber   
  \\&=
 \sum_{\substack{\ell|q^\infty,\\ \ell\le\frac{X}{2}}}\frac{c'(s)+c(s)\log\left(\frac{X}{\ell}\right)}{\ell^{s}} =
    \frac{q^s}{\varphi_s(q)}\biggl(c(s)\log X+c'(s)-c(s)\sum_{p|q}\frac{\log(p)}{p^s-1}\biggr)\nonumber
    \\&\ \ +O^*\biggl(\frac{|c(s)\log
    X+c'(s)|+|c(s)|\delta(\frac{X}{2},\sigma_0)
    \max\bigl(\log(\frac{X}{2}),\frac{1}{\sigma_0}\bigr)}{\left(\frac{X}{2}\right)^{\sigma_0}}\frac{q^{\sigma-\sigma_0}}{\varphi_{\sigma-\sigma_0}(q)}\biggr).
  \end{align}
Finally, if $s=1$, as $c(1)=\log 2$, $c'(1)=-\frac{\log^2(2)}{2}$, we derive
\begin{align*}
    \sum_{\substack{\ell\le \frac{X}{2}\\ \ell|q^\infty}}\frac{1}{\ell}
  \int_{\ell}^{2\ell} \log\left(\frac{X}{t}\right)\frac{dt}{t}&=
  \sum_{\substack{\ell|q^\infty,\\ \ell\le \frac{X}{2}}}\frac{1}{\ell}\biggl(\log(2)\log\biggl(\frac{X}{\ell}\biggr)-\frac{\log^2(2)}{2}\biggr) \\    
   &=\sum_{\substack{\ell|q^\infty,\\ \ell\le \frac{X}{2}}}\frac{c'(1)+c(1)\log\left(\frac{X}{\ell}\right)}{\ell} ,
  \end{align*}
so that the estimation \eqref{step3} holds for any $s$ with $\Re s=\sigma>\sigma_0>0$. The result is concluded by adding \eqref{step3.2} to \eqref{step3}.
\end{proof}


\begin{proofbold}[Theorem~\ref{mcheckqdex}] Let
  $K_1,K_2:D\subset\mathbb{C}\to\mathbb{C}$ be any two functions
  defined in some complex domain $D$.
  Let $h:t\mapsto K_1(s)t^{1-s}\log t+K_2(s)t^{1-s}$,
  $H:t\mapsto t^{s-1}\log t$, and the two arithmetical functions given
  by $g(n)={(-1)^{n+1}}/{n}$ 
  and $f(n)=\mathds{1}_{(n,q)=1}\mu(n)/n$.
  Then, by Lemma~\ref{coyote} and Lemma~\ref{groscoyote},
  \begin{align}\label{iddd}
    &\phantom{==}H'(t)-\frac{1}{t}\sum_{n\le t}g(n)h\left(\frac{n}{t}\right)\nonumber
    \\
    &=
    \frac{1+(s-1)\log t}{t^{2-s}}
    +\frac{K_1(s)}{t^{2-s}}\sum_{n\le t}\frac{(-1)^{n+1}}{n^{s}}\log\left(\frac{t}{n}\right)
    -\frac{K_2(s)}{t^{2-s}}\sum_{n\le t}\frac{(-1)^{n+1}}{n^{s}}\nonumber
    \\
    &=
    \frac{1+(s-1)\log t}{t^{2-s}}
    +\frac{K_1(s)}{t^{2-s}}\bigl(C(s)\log(t)+C'(s)\bigr)
    -\frac{K_2(s)}{t^{2-s}}C(s)\nonumber
    \\
    &\phantom{xxxxxxxxxxxxxxxxxxxxxx}+\ O^*\biggl(\frac{(|K_1(s)|+\sigma|K_2(s)|)(\sigma+|s|)}{\sigma^2t^2}  \biggr).
  \end{align}
  In the real case $s=\sigma>0$, for the sake of simplicity, we use
  that $1/e<1$. Thus, the real case error in Lemma~\ref{groscoyote}
  becomes $O^*(1/(\sigma X^\sigma))$ and the above factor
  $(\sigma+|s|)/\sigma$ may be replaced by $1$.  
  
  Now, by selecting $D=\mathbb{C}\setminus\{s\in\mathbb{C},s\in Z\text{ or }\zeta(s)=0\}$ and
  \begin{equation}
    \label{eq:4}
    K_1(s)=-\frac{s-1}{C(s)},\quad
    K_2(s)=\frac{C(s)+(s-1) C'(s)}{C(s)^2},
  \end{equation}
  the main term in \eqref{iddd} vanishes. As for the proof of
  Theorem~\ref{mqdex}, notice that $K_1$ and $K_2$ are well-defined.

  Moreover, by Lemma~\ref{thisconvol1} $(ii)$, we have
  $(f\star g)(n)={G_1(n)}/n$. Therefore, by
  Lemma~\ref{thisconvol1} $(iii)$ and Theorem~\ref{OFD}, we obtain
  \begin{align}
    \label{step2}
    &\sum_{\substack{n\le X,\\ (n,q)=1}}\frac{\mu(n)}{n^{s}}\log\left(\frac{X}{n}\right)
  = 
  \int^{X}_1 \biggl(\sum_{\substack{\frac{t}{2}<\ell\le t\\
      \ell|q^\infty}}\frac{1}{\ell}\biggr)
  \left(-K_1(s)\log\left(\frac{X}{t}\right)+K_2(s)\right)\frac{dt}{t^{s}}\nonumber
  \\&\phantom{xxxxxx}+X^{1-s}\int_1^X m_q\left(\frac{X}{t}\right)
  O^*\biggl(\frac{(|K_1(s)|+\sigma|K_2(s)|)(\sigma+|s|)}{\sigma^2t^2}
  \biggr)dt.
\end{align}
The first integral above can be handled with the help of Lemma
\ref{approx3}. Likewise, we can handle the second integral by
recalling estimation \eqref{M1:est}. Hence 
\begin{align}
  \label{definitive}
 \sum_{\substack{n\le X,\\ (n,q)=1}}\frac{\mu(n)}{n^{s}}\log\left(\frac{X}{n}\right)
   &=
  -\frac{q^s}{\varphi_s(q)}K_1(s)\biggl(c(s)\log
     X+c'(s)-c(s)\sum_{p|q}\frac{\log p}{p^s-1}\biggr)\nonumber&\\
   +\frac{q^s}{\varphi_s(q)}K_2(s)c(s)
   &+O^*\biggl(\frac{R_4(X;q,s)}{X^{\sigma}}
     + \frac{R_5(X;s,\sigma_0)}{X^{\sigma_0}}\frac{q^{\sigma-\sigma_0}}{\varphi_{\sigma-\sigma_0}(q)}\biggr), 
\end{align}
where
\begin{align}\label{R4:est}
R_4(X;q,s)=X\int_1^X m_q\left(\frac{X}{t}\right)
  O^*\biggl(\frac{(|K_1(s)|+\sigma|K_2(s)|)(\sigma+|s|)}{\sigma^2t^2}\biggr)dt&\nonumber\\
  =O^*\left(\frac{(|K_1(s)|+\sigma|K_2(s)|)(\sigma+|s|)}{\sigma^2}\int_1^X|m_q(t)|dt\right)&\nonumber\\
  =O^*\left(\frac{(\sigma+|s|)((\sigma+|s-1|)|C(s)|+\sigma|s-1||C'(s)|)}{\sigma^2 C(s)^2}\int_1^X|m_q(t)|dt\right)&
  \end{align}
  Note that \eqref{R4:est} allows us to define $\Xi_1(s)$ as in the statement.
  
Moreover, on recalling the definition of $R_2(X;q,s)$ and $R_3(X;q,s)$, 
\begin{equation}\label{R5:est}
R_5(X;s,\sigma_0)=|K_2(s)c(s)\zeta(s)|R_2(X;s,\sigma_0)+|K_1(s)|R_3(X;s,\sigma_0).
\end{equation}
Now, by equation \eqref{eq:4}, we immediately check that $-K_1(s)c(s)=\zeta(s)^{-1}$. Furthermore, by writing $c(s)=\frac{C(s)}{(s-1)\zeta(s)}$, we observe that
\begin{align}\label{eq::5}
  -K_1(s)
  c'(s)
  +K_2(s)c(s)
  =&\ 
    \frac{1}{C(s)\zeta(s)}\biggl(
    C'(s)-\frac{C(s)}{s-1}-\frac{C(s)\zeta'(s)}{\zeta(s)}
    \biggr)\nonumber
    \\+\frac{C(s)-(s-1)C'(s)}{C(s)}&\frac{1}{(s-1)\zeta(s)}=
    -\frac{\zeta'(s)}{\zeta^2(s)},
\end{align}
and that $R_5(X;s,\sigma_0)\leq \Xi_2(X;s,\sigma_0)$, where, by recalling \eqref{R2:est} and \eqref{R3:est},
\begin{align*}
\Xi_2(X;s,\sigma_0)=\frac{2^{\sigma_0}\Bigl(\log
  X+\delta\Bigl(\frac{X}{2},\sigma_0\Bigr)
  \max\left(\log\Bigl(\frac{X}{2}\Bigr),\frac{1}{\sigma_0}\right)\Bigr)}{|\zeta(s)|}\phantom{xxxxxxxxxxxx}&\nonumber\\
+\frac{2^{\sigma_0}\log(2)|s-1|}{|(1-2^{1-s})\zeta(s)|}+2^{\sigma_0}\left|\frac{C'(s)}{C(s)\zeta(s)}-\frac{1}{(s-1)\zeta(s)}-\frac{\zeta'(s)}{\zeta^2(s)}\right|
&\\
+2^{\sigma_0}\left|\frac{C'(s)}{C(s)\zeta(s)}-\frac{1}{(s-1)\zeta(s)}\right|+2^{\sigma_0}|e(s)|\left|\frac{C(s)+(s-1)C'(s)}{C(s)^2}\right|&.
\end{align*}
The result is concluded by noticing \eqref{eq::5} and bounding $R_5(X;s,\sigma_0)$ by $\Xi_2(X;s,\sigma_0)$ in \eqref{definitive}.

\end{proofbold}

\section{Estimates for $\Re s =\sigma\ge1$}
In this section, we specialize Theorem~\ref{mqdex} and
Theorem~\ref{mcheckqdex} to the case $\Re s=\sigma=1+\ve\ge0$ and we
derive explicit bounds. It may be necessary for some bounds in this
section to first assume $\sigma>1$ and then let $\sigma$ tend
to~$1^+$. In order to do that, we need first a series of analytic
estimations.

\subsection*{Analytic estimates}

\begin{lem}\label{inequalities}
  Let $\ve >0$ and $c(1+\ve)=\frac{1-2^{-\ve}}{\ve}$. Then, we have
    \begin{align*}
  &(i)&
  \frac{1}{\ve}<&\ \ \zeta(1+\ve)\ \ \le \frac{e^{\gamma \ve}}{\ve},&&\\
  &(ii)& \frac{1}{\log 2}<&\ \ \frac{1}{c(1+\ve)}\ \ <\frac{2^{\ve}}{\log 2},&&\\
  &(iii)& -\log 2+\frac{1}{\ve}<&\ \ \frac{\log 2}{2^\ve-1}\ \ <\frac{1}{\ve}.&&
  \end{align*}
\end{lem}

\begin{proof} $(i)$. The upper bound is found in \cite[Lemma 5.4]{Ramare*13d}. With respect to the lower bound,
  for $\sigma=1+\ve>1$, we have
  \begin{align}\label{zzeta}
    \zeta(\sigma)
    =
      \sigma\int_1^\infty \frac{[t]}{t^{\sigma+1}}dt
   & = \frac{\sigma}{\sigma-1}-\sigma\int_1^\infty \frac{\{t\}}{t^{\sigma+1}}dt
    \\&= \frac{\sigma}{\sigma-1}-1+\sigma\int_1^\infty \frac{1-\{t\}}{t^{\sigma+1}}dt>\frac{1}{\sigma-1}.\nonumber
  \end{align}
  
In order to prove $(ii)$, observe that
\begin{equation*}
  \frac{\ve}{2^\ve}<\int_0^{\ve}2^{-t}dt=\frac{1-2^{-\ve}}{\log 2}<\ve.
\end{equation*}
Thereupon, we derive $(iii)$ by observing that
\begin{equation*}
  \frac{1}{\ve}-\log 2<\frac{1}{\ve}-\frac{1-2^{-\ve}}{\ve}=\frac{1}{\ve2^\ve}<\frac{\log 2}{2^\ve-1}<\frac{1}{\ve}.
\end{equation*}
\end{proof}

\begin{lem}\label{inequalities2}
  Let $\ve >0$ and $C(1+\ve)=(1-2^{-\ve})\zeta(1+\ve)$. Then, we have
    \begin{align*}
  &(i)&\qquad
  -\frac{1}{\ve}+\frac{1}{2(1+\ve)^2}&<&\frac{\zeta'(1+\ve)}{\zeta(1+\ve)}<&\quad -\frac{1}{\ve}+2-\frac{1}{1+\ve},\\
 & (ii)&\qquad  \mathds{1}_{\ve<\frac{1}{\log 2}}\left(\frac{1}{\log 2}-\ve\right)\left(\frac{2}{e^{\gamma}}\right)^{\ve}&<&\frac{1}{C(1+\ve)}<&\quad\frac{2^{\ve}}{\log 2}, \\
    &(iii)&\qquad -\log 2+\frac{1}{2(1+\ve)^2}&<&\frac{C'(1+\ve)}{C(1+\ve)}<&\quad 2-\frac{1}{1+\ve}.
  \end{align*}
\end{lem}

\begin{proof}
Let $\sigma>1$. By \cite{Delange*87}, we have
\begin{equation}
  \label{delange}
  \frac{\zeta'(\sigma)}{\zeta(\sigma)}>-\frac{1}{\sigma-1}+\frac{1}{2\sigma^2}.
\end{equation}
On the other hand, upon multiplying by $(\sigma-1)$, we may
differentiate \eqref{zzeta} with respect to $\sigma$ and obtain 
\begin{equation*}
\zeta(\sigma)+(\sigma-1)\zeta'(\sigma)=1-(2\sigma-1)\int_1^\infty\frac{\{t\}}{t^{\sigma+1}}dt+\sigma(\sigma-1)\int_1^\infty\frac{\{t\}\log
  t}{t^{\sigma+1}}dt.
\end{equation*}
Therefore, as $2\sigma-1>0$,
\begin{align*}
\zeta(\sigma)+(\sigma-1)\zeta'(\sigma)<1+\sigma(\sigma-1)\int_1^\infty\frac{\log
  t }{t^{\sigma+1}}dt=1+\frac{\sigma-1}{\sigma},
\end{align*}
so that, by Lemma~\ref{inequalities} $(i)$,
\begin{equation*}
  \frac{\zeta'(\sigma)}{\zeta(\sigma)}<-\frac{1}{\sigma-1}+\frac{2\sigma-1}{\sigma(\sigma-1)\zeta(\sigma)}<-\frac{1}{\sigma-1}+2-\frac{1}{\sigma},
  \end{equation*}
  whence $(i)$. With respect to $(ii)$, observe that, by definition and with the help of Lemma~\ref{inequalities} $(ii)$, $(iii)$, we have
  \begin{align*}
  \mathds{1}_{\sigma<1+\frac{1}{\log 2}}\left(-1+\frac{1}{(\sigma-1)\log 2}\right)\frac{2^{\sigma-1}}{\zeta(\sigma)}&<\frac{1}{C(\sigma)}=\frac{2^{\sigma-1}}{(2^{\sigma-1}-1)\zeta(\sigma)},\\
  \frac{1}{C(\sigma)}=\frac{2^{\sigma-1}}{(2^{\sigma-1}-1)\zeta(\sigma)}&<\frac{2^{\sigma-1}}{\log(2)(\sigma-1)\zeta(\sigma)}.
  \end{align*}
  The estimation is the derived by using Lemma~\ref{inequalities} $(i)$.
  Finally, again by definition,
  \begin{equation*}
  \frac{C'(\sigma)}{C(\sigma)}=\frac{\log 2}{2^{\sigma-1}-1}+\frac{\zeta'(\sigma)}{\zeta(\sigma)}.
  \end{equation*}
  Thus, by $(i)$ and Lemma~\ref{inequalities} $(ii)$, $(iii)$, we derive $(iii)$.
\end{proof}

We provide the first part of the proof of Theorem~\ref{mqeps}. The second part, namely containing the bounds for
$\Delta_q(X,\ve)$.\\

\begin{proofbold}[Theorem~\ref{mqeps} - Part 1]
Theorem~\ref{mqdex} with $\sigma=1+\ve$, $\sigma_0=\frac{1}{2}+\ve$ gives us
\begin{multline*}
  \Biggl|
  \sum_{\substack{n\le X,\\ (n,q)=1}}\frac{\mu(n)}{n^{1+\ve}}
  -\frac{m_q(X)}{X^{\ve}}
  -\frac{q^{1+\ve}}{\varphi_{1+\ve}(q)}\frac{1}{\zeta(1+\ve)}
  \Biggr|\\
  \le
  \frac{1}{|c(1+\ve)\zeta(1+\ve)|X^{1+\ve}}\int_1^X|m_q(t)|dt  +
  \frac{|c(1+\ve)|+2^{1/2+\ve}e(1+\ve)}
  {|c(1+\ve)\zeta(1+\ve)|X^{\frac{1}{2}+\ve}}\frac{\sqrt{q}}{\varphi_{\frac{1}{2}}(q)},
\end{multline*}
 where $e(1+\ve)=2^{-\ve}(1+2^{\ve-1}\ve\log 2)\log
   2$. Further, by Lemma~\ref{inequalities}, we have
 \begin{align}
   \frac{1}{|c(1+\ve)\zeta(1+\ve)|}
   \leq&\
         \frac{\ve 2^{\ve}}{\log 2}\label{q0}\\
   \frac{|c(1+\ve)|+2^{\frac{1}{2}+\ve}e(1+\ve)}{|c(1+\ve)\zeta(1+\ve)|}
   \leq&\
         \ve(1+2^{\frac12+\ve}(1+\ve 2^{\ve-1}\log 2))\nonumber
 \end{align}
  Now, by using Lemma~\ref{boundmqdex}, we conclude that
  \begin{equation*}
  \sum_{\substack{n\le X,\\ (n,q)=1}}\frac{\mu(n)}{n^{1+\ve}}
  =\frac{m_q(X)}{X^{\ve}}
  +\frac{q^{1+\ve}}{\varphi_{1+\ve}(q)}\frac{1}{\zeta(1+\ve)}+
  \frac{\ve\Delta_q(X,\ve)}{X^{\ve}},
\end{equation*}
where
\begin{multline*}
  |\Delta_q(X,\ve)|
  \leq
  \frac{0.010333\ 2^\ve}{\log 2}
  \frac{g_1(q)\ q^\xi\ \1_{X\geq
      10^{12}}}{\varphi_{\xi}(q)\log X}
  \\
  +
   \biggl(
    1+2^{\frac12+\ve}(1+\ve 2^{\ve-1}\log 2)
    +\frac{2^\ve\sqrt{8}g_0(q)}{\log 2}\biggr) \frac{\sqrt{q}}{\varphi_{\frac{1}{2}}(q)\sqrt{X}}.
\end{multline*}
We further simplify this bound into
\begin{multline*}
  |\Delta_q(X,\ve)|
  \leq
  0.03\frac{g_1(q)\ q^\xi\ \1_{X\geq
      10^{12}}}{\varphi_{\xi}(q)\log X}  \\
  +
  \left(
    \frac{\sqrt{8}}{\log  2}g_0(q)
    +1+2^{\frac12}(1+\ve 2^{\ve-1}\log 2)
  \right)\frac{\sqrt{q}}{\varphi_{\frac12}(q)}\frac{2^\ve }{\sqrt{X}},
\end{multline*}
that is,
\begin{multline}\label{inek}
  |\Delta_q(X,\ve)|
  \leq
 0.03 \frac{g_1(q)\ q^\xi\ \1_{X\geq
      10^{12}}}{\varphi_{\xi}(q)\log X}   \\+
  \left(
    4.09\,g_0(q)
    +
    2.42 + 0.50\,\ve 2^\ve
  \right)\frac{\sqrt{q}}{\varphi_{\frac12}(q)}\frac{2^\ve}{\sqrt{X}}
\end{multline}
from which the statement of the theorem follows.

The inequality $X^{\sigma-1}m_q(X,\sigma)\ge m_q(X)$ follows by
expanding $X^{\sigma-1}m_q(X,\sigma)$ in Taylor series as in the proof
of Theorem~\ref{easy} and in using inequality \eqref{inek}. This readily implies
that $\Delta_q(X,\ve)/X^\ve\ge -q/\varphi(q)$.

\end{proofbold}


\begin{proofbold}[Theorem~\ref{mcheckqeps}]
By using Theorem~\ref{mcheckqdex} with $\sigma=1+\ve$,
$\sigma_0=\frac12+\ve$ and writing
$\Xi_2(X;\ve)=\Xi_2(X;1+\ve,\frac{1}{2}+\ve)$, we obtain 
\begin{align}\label{checkqeps}
  \check{m}_q(X;1+\ve)
  =&\nonumber
     \frac{q^{1+\ve}}{\varphi_{1+\ve}(q)}
     \biggl(\frac{\log
     X}{\zeta(1+\ve)}-\frac{\zeta'(1+\ve)}{\zeta^2(1+\ve)}-\frac{1}{\zeta(1+\ve)}\sum_{p|q}\frac{\log
     p}{p^{1+\ve}-1}
     \biggr)
  \\  
  &+\ O^*\left(\frac{\Xi_1(\ve)}{X^{1+\ve}}\int_1^X
  |m_q(t)|dt
+\frac{\Xi_2(X;\ve)}{X^{\frac{1}{2}+\ve}}\frac{\sqrt{q}}{\varphi_{\frac{1}{2}}(q)}\right).
\end{align}
Concerning $\Xi_1(\ve)$, we may reduce it to
\begin{equation*}
  \Xi_1^0(\ve) = \frac{(1+2\ve)C(1+\ve)+\ve e^{-1}|C'(1+\ve)|}{(1+\ve)C(1+\ve)^2}.
\end{equation*}
Therefore Lemma~\ref{boundmqdex} gives us
\begin{multline*}
  \check{m}_q(X;1+\ve)
  =
    \frac{q^{1+\ve}}{\varphi_{1+\ve}(q)}
    \biggl(\frac{\log
    X}{\zeta(1+\ve)}-\frac{\zeta'(1+\ve)}{\zeta^2(1+\ve)}
    -\frac{1}{\zeta(1+\ve)}\sum_{p|q}\frac{\log p}{p^{1+\ve}-1}\biggr)
  \\+ O^*\left(
    \frac{0.010333\,g_1(q)q^\xi \1_{X\geq 10^{12}} }{\varphi_\xi(q)X^\ve\log X}\Xi_1^0(\ve)
    +
    \frac{(\sqrt{8}g_0(q) \Xi_1^0(\ve)
      +\Xi_2(X;\ve))\sqrt{q}}{\varphi_{\frac12}(q)X^{\frac12+\ve} }
\right).
\end{multline*}
 As $\ve<\frac{1}{\log 2}$, we may use Lemma~\ref{inequalities2} $(ii)$, $(iii)$ and obtain
 \begin{align}\label{Xi1:est}
   |\Xi_1^0(\ve)|
   \leq&
         \frac{(1+2\ve)}{(1+\ve)|C(1+\ve)|}+\frac{\ve|C'(1+\ve)|}{(1+\ve)
         C(1+\ve)^2}\nonumber
   \\\leq&
           \frac{2^{\ve}(1+2\ve)}{(1+\ve)\log 2}
           + \frac{2^{\ve}\ve}{(1+\ve)\log 2}\max\left(2-\frac{1}{1+\ve},\log
           2-\frac{1}{2(1+\ve)^2}\right)\nonumber
   \\=&\frac{2^{\ve}}{\log 2}\left(2-\frac{1}{1+\ve}\right)^2,
\end{align}
where we have used that $2-\log 2>1>\frac{1}{1+\ve}-\frac{1}{2(1+\ve)^2}$.

On the other hand, we can bound
$\Xi_2(X;1+\ve,\frac{1}{2}+\ve)=\Xi_2(X;\ve)$ by noticing that,
as $X\geq 15$ and $\ve\in(0,\frac{1}{2}]$,
$\max\left(\log\left(\frac{X}{2}\right),\frac{1}{\frac{1}{2}+\ve}\right)=\log\left(\frac{X}{2}\right)<\log X$. 
Thus 
\begin{align}
\label{Xi2.1}
  |\Xi_2(X;\ve)|
  \leq&
        \frac{2^{\frac{3}{2}+\ve}\log
        X}{|\zeta(1+\ve)|}+\frac{2^{\frac{1}{2}+\ve}}{|\zeta(1+\ve)|}\left|\frac{C'(1+\ve)}{C(1+\ve)}-\frac{1}{\ve}-\frac{\zeta'(1+\ve)}{\zeta(1+\ve)}\right|\nonumber
  \\ &+\frac{2^{\frac{1}{2}+\ve}\log 2}{|c(1+\ve)\zeta(1+\ve)|}+\frac{2^{\frac{1}{2}+\ve}}{|\zeta(1+\ve)|}\left|\frac{C'(1+\ve)}{C(1+\ve)}-\frac{1}{\ve}\right|
       \nonumber
  \\&+2^{\frac{1}{2}+\ve}\frac{e(1+\ve)}{|C(1+\ve)|}\left|1+\frac{\ve C'(1+\ve)}{C(1+\ve)}\right|,
\end{align}  
where $e(1+\ve)=\log 2+2^{\ve-1}\log^2(2)\ve$. In order to further
estimate \eqref{Xi2.1}, by recalling the definition of $C(1+\ve)$ and
on using Lemma~\ref{inequalities} $(iii)$, we have 
\begin{align}\label{q1}
\left|\frac{C'(1+\ve)}{C(1+\ve)}-\frac{1}{\ve}-\frac{\zeta'(1+\ve)}{\zeta(1+\ve)}\right|=\left|\frac{\log 2}{2^\ve-1}-\frac{1}{\ve}\right|<\log 2.
\end{align}
Also, by \eqref{q1} and Lemma~\ref{inequalities2} $(i)$, we have
\begin{align}\label{q2}
	\left|\frac{C'(1+\ve)}{C(1+\ve)}-\frac{1}{\ve}\right|&=\left|\frac{\log 2}{2^\ve-1}-\frac{1}{\ve}+\frac{\zeta'(1+\ve)}{\zeta(1+\ve)}\right|<\left|\frac{\log 2}{2^\ve-1}-\frac{1}{\ve}\right|+\left|\frac{\zeta'(1+\ve)}{\zeta(1+\ve)}\right|\nonumber\\
	&\phantom{xxxxxxxxxxxxxxxxxxxx}<\frac{1}{\ve}+\log 2-\frac{1}{2(1+\ve)^2}
\end{align}
where we have used that $2<\frac{1}{\ve}+\frac{1}{1+\ve}$.
So, by Lemma~\ref{inequalities2}, $(iii)$, we have
\begin{align}\label{q3}
\left|1+\frac{\ve C'(1+\ve)}{C(1+\ve)}\right|&<\left(1+\ve\max\left(2-\frac{1}{1+\ve},\log 2-\frac{1}{2(1+\ve)^2}\right)\right)\nonumber\\
&=1+2\ve-\frac{\ve}{1+\ve}
\end{align}
where we have used that $2-\log 2>\frac{1}{1+\ve}$.
Subsequently, by using Lemma~\ref{inequalities} $(i)$ and Lemma
\ref{inequalities2} $(ii)$ and putting \eqref{q1}, \eqref{q0},
\eqref{q2} and \eqref{q3} together with \eqref{Xi2.1}, we obtain  
\begin{align*}
  |\Xi_2(X;\ve)|2^{-\ve}
  \leq&
        \ve 2^{\frac{3}{2}}\log X
        + \ve 2^{\frac{1}{2}}\log 2
        + 2^{\frac{1}{2}}\biggl(1+\ve\log 2-\frac{\ve}{2(1+\ve)^2}\biggr)
  \\ &+\ve 2^{\frac{1}{2}+\ve}
       + 2^{\frac{1}{2}+\ve}\biggl(1+\frac{\ve 2^{\ve}\log  2}{2}\biggr)
       \biggl(1+2\ve-\frac{\ve}{1+\ve}\biggr)
  \\\leq&
          2.93 + 2.83 \ve\log X +5.17\ve
\end{align*}
where, in the last line, we have used $\ve \le 1/10$.
\end{proofbold}

\section{Bounding $\Delta_q(X,\ve)$ from above}\label{discrete}
\label{SignDelta}

Notice that
\begin{equation}\label{Delta}
  \Delta_q(X,\ve)
  =\sum_{\substack{n\le X\\
      (n,q)=1}}\frac{\mu(n)}{n}\frac{(X/n)^\ve-1}{\ve}
  -\frac{q^{1+\ve}}{\varphi_{1+\ve}(q)}\frac{X^\ve}{\ve\zeta(1+\ve)}.
\end{equation}
Our aim is to study the above quantity algorithmically for small values
of the parameters $X$ and $q$ and for varying $\ve\in[0,1]$.
When $q$ is squarefree, an important simplification occurs: if $q'|q$ and the only
prime factors of $q/q'$ are (strictly) larger than $X$, then
$\Delta_q(X,\ve)\le \Delta_{q'}(X,\ve)$. Thus, in bounding $\Delta_q(X,\ve)$ from above, it suffices to restrict our attention to values of $q$ whose prime factors do not exceed $X$.

\subsection*{Discretising in $\ve$}
We start with a (rough) bound for the derivative of
$\Delta_q(X,\ve)$ with respect to~$\ve$. This result allows us to build an algorithm that discretises the variable $\ve$, which is then used to obtain the second part of  Theorem~\ref{mqeps}.
\begin{lem}
  \label{derDelta}
For any $X\geq 1$, $\ve\in[0,1]$ and $q\in\mathbb{Z}_{>0}$, we have the following inequalities
  \begin{align*}
&\frac{\varphi(q)}{q}\frac{\mathrm{d}}{\mathrm{d}\ve}\frac{\Delta_q(X,\ve)}{X^\ve}
\leq\log X+\sum_{p|q}\frac{\log p}{p-1}-\frac{1}{2\ve(1+\ve)^2\zeta(1+\ve)},\\
&-\log X-\frac{1+2\ve}{\ve(1+\ve)\zeta(1+\ve)}
\leq\frac{\varphi(q)}{q}\frac{\mathrm{d}}{\mathrm{d}\ve}\frac{\Delta_q(X,\ve)}{X^\ve}.
\end{align*}
\end{lem}
\begin{proof} By \eqref{Delta}, we have
\begin{equation*}
\frac{\Delta_q(X,\ve)}{X^{\ve}}=\frac{X^\ve m_q(X,1+\ve)-m_q(X)}{\ve\ X^\ve}-\frac{q^{1+\ve}}{\varphi_{1+\ve}(q)}\frac{1}{\ve\ \zeta(1+\ve)}.
\end{equation*}
  On the other hand, we have
  \begin{align}\label{der1}
    \frac{X^\ve m_q(X,1+\ve)-m_q(X)}{\ve}&=\sum_{\substack{n\le X\\
        (n,q)=1}}\frac{\mu(n)}{n}\frac{(X/n)^\ve-1}{\ve}\nonumber\\
    &=\sum_{k\ge1}\frac{\ve^{k-1}}{k!}\sum_{\substack{n\le X\\
        (n,q)=1}}\frac{\mu(n)}{n}\log^k\biggr(\frac{X}{n}\biggr),
  \end{align}
  where the exchange of summations has occurred thanks to absolute convergence. By Lemma~\ref{Prim}, \eqref{der1} is non-negative and upper bounded by
   \begin{equation*}
     \sum_{k\ge1}\frac{\ve^{k-1}}{k!}k\frac{q}{\varphi(q)}(\log X)^{k-1}
     =\frac{q}{\varphi(q)}X^\ve.
   \end{equation*}
  Moreover, the derivative of the expression \eqref{der1} with respect to $\ve$ reads
  \begin{equation*}
    \sum_{k\ge2}\frac{(k-1)\ve^{k-2}}{k!}\sum_{\substack{d\le X\\
        (d,q)=1}}\frac{\mu(d)}{d}\log^k\biggl(\frac{X}{d}\biggr),
  \end{equation*}
  which, again by Lemma~\ref{Prim}, is non-negative and bounded from above by
  \begin{equation*}
    \sum_{k\ge2}\frac{(k-1)\ve^{k-2}}{(k-1)!}\frac{q}{\varphi(q)}\log^{k-1}(
    X)
    =\frac{q}{\varphi(q)}X^\ve\log X.
  \end{equation*}
  We then conclude that
  \begin{align}\label{B1}
    \frac{\mathrm{d}}{\mathrm{d}\ve}\left(\frac{X^\ve m_q(X,1+\ve)-m_q(X)}{\ve\ X^\ve}\right)&\in\left[0,\frac{q}{\varphi(q)}\log X\right]-\log(X)\left[0,\frac{q}{\varphi(q)}\right]\nonumber\\
    &\in\left[-\frac{q}{\varphi(q)}\log X,\frac{q}{\varphi(q)}\log X\right].
  \end{align}
    On the other hand, by recalling \eqref{diff}, and using the chain rule, we obtain
  \begin{equation*}
  \frac{\text{d}}{\text{d}\ve}\frac{q^{1+\ve}}{\varphi_{1+\ve}(q)}
  =
  -\frac{q^{1+\ve}}{\varphi_{1+\ve}(q)}\sum_{p|q}\frac{\log p}{p^{1+\ve}-1}.
\end{equation*}
  Further, thanks to Lemma~\ref{inequalities2} $(i)$, we compute that
  \begin{equation*}
    \frac{\mathrm{d}}{\mathrm{d}\ve}\left(\frac{-1}{\ve\zeta(1+\ve)}\right)
    =\frac{\frac{1}{\ve}+\frac{\zeta'(1+\ve)}{\zeta(1+\ve)}}{\ve\zeta(1+\ve)}
    \in \biggl[\frac{1}{2\ve(1+\ve)^2\zeta(1+\ve)}, \frac{1+2\ve}{\ve(1+\ve)\zeta(1+\ve)}\biggr].
  \end{equation*}
  Therefore, we have
   \begin{align}\label{B2}
 & \frac{\text{d}}{\text{d}\ve}\left(-\frac{q^{1+\ve}}{\varphi_{1+\ve}(q)}\frac{1}{\ve\zeta(1+\ve)}\right)
  =
  \frac{q^{1+\ve}}{\varphi_{1+\ve}(q)}\frac{\sum_{p|q}\frac{\log p}{p^{1+\ve}-1}-\left(\frac{1}{\ve}+\frac{\zeta'(1+\ve)}{\zeta(1+\ve)}\right)}{\ve\zeta(1+\ve)}\nonumber\\
  &\in\frac{q^{1+\ve}}{\varphi_{1+\ve}(q)}\ \Bigg[-\frac{1+2\ve}{\ve(1+\ve)\zeta(1+\ve)},\sum_{p|q}\frac{\log p}{p-1}-\frac{1}{2\ve(1+\ve)^2\zeta(1+\ve)}\Bigg]
\end{align}
 where we have used that
  $1\le \ve\zeta(1+\ve)$, by Lemma~\ref{inequalities} $(i)$. Finally, by using that $\frac{q^{1+\ve}}{\varphi_{1+\ve}(q)}\le \frac{q}{\varphi(q)}$ and putting \eqref{B1} and \eqref{B2} together, we conclude the result.
\end{proof}
\begin{proofbold}[Theorem~\ref{mqeps} - Part 2]
  The proof has two stages: we discretise in~$X$, and then we build an algorithm.
\subsection*{\phantom{xx}Discretising in $X$}
For any positive integer $N$ and $\ve>0$, we have
\begin{multline}\label{maxx}
  \max_{N\le X<N+1}\frac{\Delta_q(X,\ve)}{X^{\ve}}=
  \frac{m_q(N, 1+\ve)}{\ve}
  \\+\frac{1}{\ve}\max\biggl(\frac{-m_q(N)}{N^\ve},\frac{-m_q(N)}{(N+1)^\ve}\biggr)
  -\frac{q^{1+\ve}}{\varphi_{1+\ve}(q)}\frac{1}{\ve\zeta(1+\ve)},
\end{multline}
the maximum depending or whether or not $m_q(N)\ge0$.

Further, at $\ve=0$, we have
\begin{align*}
  \Delta_q(X,0)=\lim_{\ve\to 0^+}\Delta_q(X,\ve)
  &=\left.\frac{\mathrm{d}}{\mathrm{d}\ve}(X^\ve
    m_q(X,1+\ve))\right|_{\ve=0}-\lim_{\ve\to
    0^+}\frac{q^{1+\ve}}{\varphi_{1+\ve}(q)}\frac{X^\ve}{\ve\
    \zeta(1+\ve)}
  \\&=[X^\ve \check{m}_q(X,1+\ve)|_{\ve=0}-\frac{q}{\varphi(q)}
  \\&=\sum_{\substack{d\le X\\
  (d,q)=1}}\frac{\mu(d)}{d}\log\left(\frac{X}{d}\right)-\frac{q}{\varphi(q)}
  \\&=
  \check{m}_q([X])-\frac{q}{\varphi(q)}+m_q([X])\log\left(\frac{X}{[X]}\right).
\end{align*}

\subsection*{\phantom{xx}Algorithm}
The initial data of this algorithm is a threshold $X_0>0$.
The points $N$ enable us to build a Pari/GP script to determine
whether $\Delta_q(X,\ve)\le 0$ for all $\ve\in[0,1]$ and $X\leq X_0$
for some $X_0>0$. Let $\ve_0=0$ and $N=1$. 
\begin{enumerate}
  \item For every divisor of $\prod_{p\le X_0}p$, run the next
    process starting with $N=1$ and $\ve_0=0$.
    \begin{enumerate}
    \item Treat the case $X\in[N,\min(N+1,X_0))$.
    \item Determine a uniform upper bound $M$ for the derivative with respect to $\ve$ of
      $X^{-\ve}\Delta_q(X,\ve)$ via Lemma~\ref{derDelta}, for $\ve\in[0,1]$.
    \item Compute $\check{m}_q(N)$ and $m_q(N)$.
    \item Compute the maximum $t_0$ of $\check{m}_q(N)-\frac{q}{\varphi(q)}$ and
      $\check{m}_q(N)-\frac{q}{\varphi(q)}+m_q(N)\log\left(\frac{N+1}{N}\right)$,
      depending on whether or not $m_q(N)\geq 0$. 
    \item It $t_0\ge0$, exit with value \texttt{FAIL}.
    \item If $t_0<0$, set $\ve_1=\ve_0-t_0/M=-t_0/M$.  Indeed, by the
      mean value theorem, for any $\ve^*\in[\ve_0,\ve_1]$,
      $\Delta_q(X,\ve^*)X^{-\ve^*}\leq M\ve_1+t_0\leq 0$, so
      $\Delta_q(X,[\ve_0,\ve_1])\leq 0$.
    \item Continue until $\ve_k\ge 1$:
      \begin{enumerate}
      \item Compute $t_{k}=\max_{N\le X<N+1}\Delta_q(X,\ve_k)X^{-\ve_k}$ using \eqref{maxx}.
      \item If $t_{k}\ge 0$, exit with value \texttt{FAIL}.
      \item If $t_{k}<0$, set $\ve_{k+1}=\ve_k-t_{k}/M$.
      \end{enumerate}
    \item Replace $N$ by $N+1$.
    \end{enumerate}
    \item When we reach this point, we select another divisor of
      $\prod_{p\le X_0}p$.
    \end{enumerate}
    If running this algorithm ends without the return value
    \texttt{FAIL}, it proves that $\Delta_q(X,\ve)\le 0$ for every
    $X\le X_0$ and every $\ve\in[0,1]$. We may easily adapt the
    script to avoid some hand-selected values of~$q$.
    
This algorithm works when the values of $\Delta_q(X,\ve)$ denoted by
$t_k$ are negative and far enough from~0. Nonetheless, it fails at $q=1$ several
times because, in fact, $\Delta_1(X,0)\ge 0$ often as $X$ changes. Moreover, for such values
of~$X$, while trying to bootstrap the algorithm at a latter value of
$\ve$, we see that $\Delta_1(X,\ve)$ remains non-negative whenever $\ve\in[0,1]$.
\end{proofbold}

\bigskip
\appendix

\section*{Appendix · A classical inequality for the primes}

We shall prove Theorem~\ref{Harmonique}, which follows thanks to Lemma~\ref{230914b} by using a similar identity approach to that of Theorem \ref{OFD}.

\subsection*{An integral identity}
The function $\alpha$ and $\beta$ are defined respectively in~\eqref{defalpha} and in~\eqref{defbeta}.
Let $\vp :\, ]0, \infty[ \rightarrow \mathbb{R}$ be a locally integrable function that vanishes in a neighbourhood of $0$. Then the function $S_{\un}\vp$, defined over $(0,\infty)$ as
\[
S_{\un}\vp(x)=\sum_n \vp(x/n) \quad (x>0),
\]
satisfies the same conditions as $\varphi$.
\begin{lem}\label{230914b} On considering the above conditions, we have
\[
\int_0^X \vp(t) \, \frac{dt}{t^2}= \frac{2}{X^2}\int_0^X S_{\un}\vp(t) \, dt-\int_0^X \vp(t) \, \alpha\biggl(\frac{X}{t}\biggr) \, \frac{dt}{t^2}, \qquad X>0.
\]
\end{lem}
\begin{proof}
Thanks to \eqref{230910a}, we have
\begin{align*}
  \int_0^X S_{\un}\vp(t) \, dt
  &= \int_0^X \Big(\sum_n \vp\biggl(\frac{t}{n}\biggr)\Big) \, dt 
=\sum_n \int_0^X  \vp\biggl(\frac{t}{n}\biggr) \, dt \\
&=\sum_n n\int_0^{\frac{X}{n}}  \vp(t) \, dt
=\int_0^X \vp(t)\Big(\sum_{n \le \frac{X}{t}}n\Big) \, dt\\
&=\frac{X^2}2\int_0^X \vp(t) \biggl(1+\alpha\biggl(\frac{X}{t}\biggr)\biggr)\, \frac{dt}{t^2}.
\end{align*}
\end{proof}
\subsection*{Analysis on two functions}
Recall the definitions of $\alpha$ and $\beta$ in~\eqref{defalpha} and in~\eqref{defbeta}, respectively. We may visualize as follows: the function $\beta$ is red colored, and its right derivate, the function $\alpha$ is blue colored. 
\begin{center}
\includegraphics[width=8.5cm]{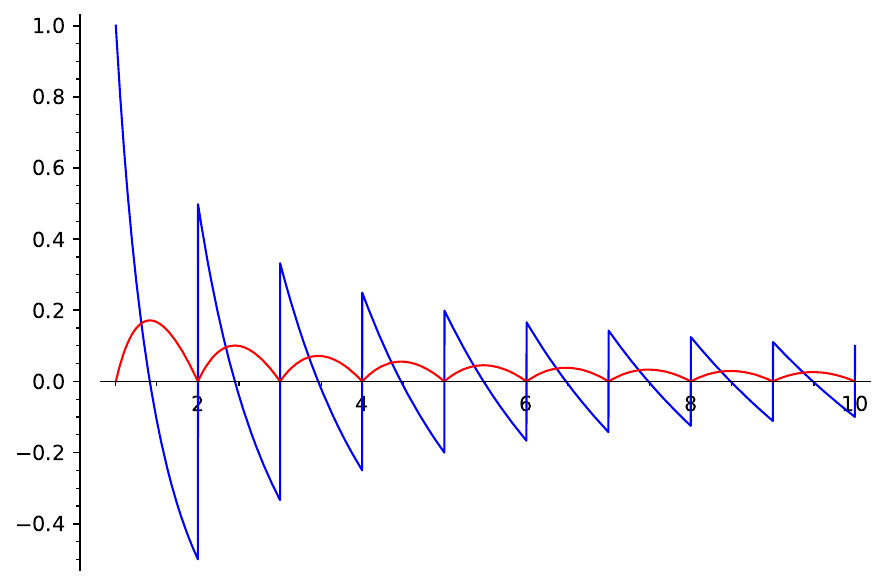}
\end{center}

 By setting $t=\set{x}\in[0,1)$, we have over any interval $[k,k+1)$, $k\in\mathbb{Z}_{>0}$, that
\begin{align*}
\beta(x)&=\frac{\set{x}-\set{x}^2}x=\frac{t-t^2}{t+k}, \\
\alpha(x) &=\frac{1-2\set{x}}x-\frac{\set{x}-\set{x}^2}{x^2} =\frac{1-2t}{t+k}-\frac{t-t^2}{(t+k)^2}.
\end{align*}

Hence
\[
x^2\alpha(x)=(t+k)(1-2t)+t^2-t=k^2+k-(t+k)^2.
\]

Moreover, by defining $t_k=\sqrt{k^2+k}\, -k$, which is a real number between $0$ and $1/2$, we derive 
\[
\alpha(x) > 0, \quad \text{ if }0\le t < t_k \epv \alpha(x) < 0, \quad \text{ if }t_k < t <1.
\]

\smallskip

On the other hand, observe that
\begin{align*}
\int_{k+t_k}^{k+1} \abs{\alpha(x)}\, \frac{dx}x &=- \int_{t_k}^1\alpha(t+k)\,\frac{dt}{t+k}\\
&=-\left[\frac{\beta(t+k)}{t+k}\right|_{t_k}^1-\int_{t_k}^1\beta(t+k)\frac{dt}{(t+k)^2}\\
&=\frac{t_k-t_k^2}{k(k+1)}-\int_{t_k}^1(t-t^2)\frac{dt}{(t+k)^3},
\end{align*}
where
\begin{align*}
\int_{t_k}^1(t-t^2)\frac{dt}{(t+k)^3} &=- \left[\frac{t-t^2}{2(t+k)^2}\right|_{t_k}^1+\int_{t_k}^1 (1-2t)\frac{dt}{2(t+k)^2}\\
&=\frac{t_k-t_k^2}{2k(k+1)} - \left[\frac{1-2t}{2(t+k)}\right|_{t_k}^1-\int_{t_k}^1 \frac{dt}{t+k}\\
&=\frac{t_k-t_k^2}{2k(k+1)}+\frac{1}{2(k+1)}+\frac{1-2t_k}{2\sqrt{k(k+1)}}-{\frac12}\log\biggl(1+\frac{1}{k}\biggr).
\end{align*}
Thus,
\begin{equation*}
\int_{k+t_k}^{k+1} \abs{\alpha(x)}\, \frac{dx}x=\frac{t_k-t_k^2}{2k(k+1)}-\frac{1-2t_k}{2\sqrt{k(k+1)}}+{\frac12}\log\biggl(1+\frac{1}{k}\biggr)-\frac{1}{2(k+1)}.
\end{equation*}
Now,
\begin{align*}
  \frac{t_k-t_k^2}{2k(k+1)}-\frac{1-2t_k}{2\sqrt{k(k+1)}}
  &=\frac{\sqrt{k(k+1)}-k-k(k+1)-k^2+2k\sqrt{k(k+1)}}{2k(k+1)}
    \\&\quad -\frac{1-2\sqrt{k(k+1)}+2k}{2\sqrt{k(k+1)}}
  \\
  &=\frac{(2k+1)\sqrt{k(k+1)}}{2k(k+1)}-1-\frac{2k+1}{2\sqrt{k(k+1)}}+1
    =0.
\end{align*}
Therefore,
\begin{equation*}
\int_{k+t_k}^{k+1} \abs{\alpha(x)}\, \frac{dx}x={\frac12}\log\biggl(1+\frac{1}{k}\biggr)-\frac{1}{2(k+1)},
\end{equation*}
and we deduce that
\begin{align*}
2\int_{1}^{\infty} \abs{\alpha(x)}\1_{\alpha(x)<0}\, \frac{dx}x &=\sum_{k=1}^{\infty}\Big(\log\biggl(1+\frac{1}{k}\biggr)-\frac{1}{k+1}\Big)\\
&=\lim_{K \rightarrow \infty} \Big(\sum_{k=1}^K\log\biggl(1+\frac{1}{k}\biggr)-\sum_{k=1}^K\frac{1}{k+1}\Big)\\
&=1-\gamma.
\end{align*}
We have just proved the following lemma.
\begin{lem}
  \label{230918a}
  We have $\displaystyle \int_{1}^{\infty} \abs{\alpha(x)}\1_{\alpha(x)<0}\, \frac{dx}x =\frac{1-\gamma}2$.
\end{lem}

The reader should compare the above equality with the following
\begin{align*}
  \int_{1}^{\infty} \alpha(x)\, \frac{dx}x
  &= \left[\frac{\beta(x)}{x}\right|_{1}^{\infty}+\int_1^{\infty}
    \beta(x)\frac{dx}{x^2}
  =\int_1^{\infty} (\set{x}-\set{x}^2)\frac{dx}{x^3}
  \\&=-\left[\frac{\set{x}-\set{x}^2}{2x^2}\right|_{1}^{\infty}+\int_1^{\infty}
  (1-2\set{x})\frac{dx}{2x^2}
  =\gamma-{\frac12},
\end{align*}
which for instance follows from de \cite[Eq. (10)]{Balazard*16}.

\subsection*{Integrating the Stirling formula}

Consider the Stirling formula in the version given in \cite[p.17]{Balazard*16}
\[
\sum_{n \le t} \log n=t\log t-t +\biggl(\frac{1}{2}-\set{t}\biggr)\log t + \gamma_{0,1}+\ve_{0,1}(t), \quad t>0,
\]
where
\begin{align*}
\gamma_{0,1} &= 1+\int_1^{\infty}(\set{u}-1/2) \, \frac{du}u=\frac{\log 2\pi}2,\\
\ve_{0,1}(t)  &=\int_t^{\infty}(1/2-\set{u}) \, \frac{du}u.
\end{align*}

For any $X>0$, we then have
\begin{align}\label{3int}
&\phantom{xxxxxxxxxx}\int_0^X\Big(\sum_{n \le t} \log n \Big) \, dt=\nonumber\\
& \int_0^X(t\log t-t )\, dt +\int_0^X \biggl(\frac{1}{2}-\set{t}\biggr)\log t\, dt +\gamma_{0,1}X+\int_0^X \ve_{0,1}(t)\, dt.
\end{align}

We shall calculate the above three integrals. The first integral satisifies
\[
\int_0^X(t\log t-t )\, dt =\frac{X^2}2\biggl(\log X -\frac{3}{2}\biggr), \quad X>0.
\]
As for the second integral in \eqref{3int}, it may be expressed as
\[
  \int_0^X \biggl(\frac{1}{2}-\set{t}\biggr)\log t\, dt
  =\int_0^1 \biggl(\frac12-t\biggr)\log t\, dt
  +\int_1^X \biggl(\frac12-\set{t}\biggr)\log t\, dt,
\]
where 
\[
  \int_0^1 \biggl(\frac{1}{2}-t\biggr)\log t\, dt
  = \left[\frac{t-t^2}2\log t\right|_0^1-\int_0^1 \frac{1-t}2\, dt=-\frac 14.
\]

Finally, the third integral in \eqref{3int} may be written as
\[
\int_0^X \ve_{0,1}(t)\, dt = X\ve_{0,1}(X)+\int_0^X
\biggl(\frac12-\set{t}\biggr)\, dt
= X\ve_{0,1}(X)+\frac{\set{X}-\set{X}^2}2.
\]

Subsequently,
\begin{multline}\label{230914a}
\int_0^X\Big(\sum_{n \le t} \log n \Big) \, dt= \frac{X^2}2\biggl(\log X -\frac32\biggr)
\\+\gamma_{0,1}X+\int_1^X \biggl(\frac12-\set{t}\biggr)\log t\, dt-\frac
14+X\ve_{0,1}(X)+\frac{\set{X}-\set{X}^2}2. 
\end{multline}

\begin{proofbold}[Theorem~\ref{Harmonique}]
Apply Lemma~\ref{230914b} to the following function
\[
\vp(X)=\psi(X)=\sum_{n\le X} \Lambda(n),\quad X>0.
\]
Thus, for any $X>0$,
\[
\int_0^X \psi(t) \, \frac{dt}{t^2}= \frac{2}{X^2}\int_0^X S_{\un}\psi(t) \, dt-\int_0^X \psi(t) \, \alpha\biggl(\frac{X}{t}\biggr) \, \frac{dt}{t^2},
\]
so that
\begin{equation}\label{inn}
\sum_{n \le X}\Lambda(n)\Big(\frac 1n - \frac 1X\Big)=\frac{2}{X^2}\int_0^X \Big(\sum_{n \le t} \log n \Big) \, dt -\int_0^X \psi(t) \, \alpha\biggl(\frac{x}{t}\biggr) \, \frac{dt}{t^2}.
\end{equation}

Hence, by putting \eqref{230914a} into \eqref{inn}, we obtain
\[
\sum_{n \le X} \frac{\Lambda(n)}n =\log X + f(X),
\]
where
\begin{align}\label{f(X)}
f(X)=&\frac{\psi(X)}X -\frac32 -\int_0^X \psi(t) \, \alpha\biggl(\frac{X}{t}\biggr) \, \frac{dt}{t^2}
+\frac{2\gamma_{0,1}}X \nonumber
\\&+ \frac{2}{X^2}\int_1^X \biggl(\frac12-\set{t}\biggr)\log t\, dt-\frac 1{2X^2} +\frac{2\ve_{0,1}(X)}X+\frac{\set{X}-\set{X}^2}{X^2}.
\end{align}

In order to conclude the result, we just need to show that $f(X) \le
0$ for any $X \ge 1$. Observe first that $\abs{\ve_{0,1}(X)} \le
{1}/(8X)$ by~\cite[p.17]{Balazard*16} and that
\begin{equation*}
  \int_1^X\biggl(\frac12-\set{t}\biggr)\log t\, dt
  \le \frac{\log X}8
\end{equation*}
by the second mean value theorem, in the classical version given for
instance in
Section~12.3 of~\cite{Titchmarsh*32}. 
Thus, for any $X\ge1$,
\begin{align}\label{ii1}
\frac{2\gamma_{0,1}}X + \frac{2}{X^2}\int_1^X \biggl(\frac12-\set{t}\biggr)\log t\,
dt-\frac 1{2X^2} 
&+\frac{2\ve_{0,1}(X)}X+\frac{\set{X}-\set{X}^2}{X^2}\nonumber\\
&\le \frac{2\gamma_{0,1}}X+\frac{\log X}{4X^2}.
\end{align}

\smallskip

On the other hand, for the terms appearing in \eqref{f(X)} that involve the function $\psi$, we are going to use Hanson's inequality~\cite{Hanson*72}, namely
\begin{equation}\label{ii2}
\psi(X) \le X\log 3, \quad X\ge 1.
\end{equation}
Thereupon, we observe that
\begin{align}\label{ii3}
  -\int_0^X \psi(t) \, \alpha\biggl(\frac{X}{t}\biggr) \, \frac{dt}{t^2}
  &\le \int_0^X \psi(t) \, \abs{\alpha\biggl(\frac{X}{t}\biggr)}\1_{\alpha(X/t) <0} \, \frac{dt}{t^2}\nonumber\\
&\le (\log 3)\int_0^x  \, \abs{\alpha\biggl(\frac{X}{t}\biggr)}\1_{\alpha(X/t) <0} \, \frac{dt}{t}\nonumber\\
&= (\log 3)\int_{1}^{\infty} \abs{\alpha(u)}\1_{\alpha(u)<0}\, \frac{du}u
=\frac{(1-\gamma)\log 3}2,
\end{align}
thanks to Lemma~\ref{230918a}.
All in all, by putting \eqref{ii1}, \eqref{ii2} and \eqref{ii3} together, and recalling the definition \eqref{f(X)}, we deduce the following inequality
\[
f(X)\le \log 3-\frac 32+\frac{(1-\gamma)\log 3}2
+\frac{2\gamma_{0,1}}X+\frac{\log X}{4X^2}
=f_1(X),\quad X \ge 1.
\]

Moreover, we have
\[
f_1'(X)=\frac{1}{4X^3}-\frac{\log(2\pi)}{X^2}-\frac{\log X}{4X^3} <0,\quad X\ge 1,
\]
so that $f_1$ is strictly decreasing in $[1,\infty)$. Since $f_1(12)=-0.011679\dots$, we derive 
\[
f(X) \le f_1(X) <0, X\ge 12.
\]

Finally, it is sufficient to see that
\[
\sum_{n \le X} \frac{\Lambda(n)}n \le \log X, \qquad 1 \le X\le 12,
\]
which is indeed true for $X=2,3,4,5,7,8,9,11$.
\end{proofbold}

\medskip

If we used an inequality of the form  $\psi(X) \le aX$ instead of Hanson's, we would have obtained Theorem~\ref{Harmonique} for $X\ge
X_0(a)$, for some $X_0(a)>0$, provided that
\[
a< \frac{3}{3-\gamma}=1.23824\dots\ .
\]

\bibliographystyle{plain}
\bibliography{Local}

\end{document}